\documentclass[a4paper,11pt]{article}

\usepackage{graphicx}
\usepackage{amsmath}
\usepackage{amsfonts}
\usepackage{titlesec}
\usepackage{amsthm}
\usepackage{amssymb}
\usepackage{geometry}
\usepackage{color}
\usepackage{accents}
\usepackage{epic}

\usepackage{algorithm2e}

\newcommand{\R}{\mathbb{R}}
\newcommand{\C}{\mathbb{C}}

\newcommand{\N}{\mathbb{N}}

\newcommand{\V}{H^1_0(\mathcal{D})}

\newcommand{\ci}{\mathrm{i}} 
\newcommand{\half}{\frac{1}{2}}

\newcommand{\PR}{P_{h}}

\newcommand{\GM}{G_{\hspace{-2pt}\mbox{\tiny$M$}}}

\newcommand{\hatu}{u_{\deltat{}}}
\newcommand{\hatuh}{u_{h,\deltat{}}}
\newcommand{\ham}{\mathcal{H}}

\newcommand{\hatuM}[1]{u_{\deltat{}}^{M,#1}}
\newcommand{\hatuhM}[1]{u_{h,\deltat{}}^{M,#1}}

\newcommand{\hateM}[1]{e_{\deltat{}}^{M,#1}}
\newcommand{\hatehM}[1]{e_{h,\deltat{}}^{M,#1}}

\newcommand{\GammaF}{F}

\newcommand{\Sh}{S_h}
\newcommand{\eps}{\varepsilon}

\newcommand{\Ltwo}[2]{\langle #1 , #2 \rangle_{L^2({\mathcal{D}})}}
\newcommand{\Ltwobig}[2]{\biggl\langle #1 , #2 \biggr\rangle_{\hspace{-1ex}L^2({\mathcal{D}})}}

\newcommand{\deltat}[1]{\tau_{#1}}

\newcommand{\quotes}[1]{``#1''}

\definecolor{dark-green}{rgb}{0.0,0.4,0.0}

\newtheorem{theorem}{Theorem}[section]

\newtheorem{lemma}[theorem]{Lemma}

\theoremstyle{definition}
\newtheorem{definition}[theorem]{Definition}
\newtheorem{remark}[theorem]{Remark}
\usepackage{framed}

\newenvironment{fshaded}{%
\MakeFramed {\FrameRestore}}%
{\endMakeFramed}

\newtheorem{cstep}{Step}

\begin{document}

\begin{center}
{\LARGE Crank-Nicolson Galerkin approximations to nonlinear Schr\"odinger equations with rough potentials\renewcommand{\thefootnote}{\fnsymbol{footnote}}\setcounter{footnote}{0}
 \hspace{-3pt}\footnote{P. Henning acknowledges funding by the Swedish Research Council (grant 2016-03339) and D. Peterseim acknowledges support by the Institute for Numerical Simulation at the University of Bonn, by the Hausdorff Center for Mathematics Bonn, and by Deutsche Forschungsgemeinschaft in the Priority Program 1748 \quotes{Reliable simulation techniques in solid mechanics} (PE2143/2-1).}}\\[2em]
\end{center}

\renewcommand{\thefootnote}{\fnsymbol{footnote}}
\renewcommand{\thefootnote}{\arabic{footnote}}

\begin{center}
{\large Patrick Henning\footnote[1]{Department of Mathematics, KTH Royal Institute of Technology, SE-100 44 Stockholm, Sweden.},
Daniel Peterseim\footnote[2]{Institut f\"ur Mathematik, Universit\"at Augsburg, Universit\"atsstr. 14, 86159 Augsburg, Germany}}\\[2em]
\end{center}

\begin{center}
{\large{\today}}
\end{center}

\begin{center}
\end{center}

\begin{abstract}
This paper analyses the numerical solution of a class of non-linear Schr\"odinger equations by Galerkin finite elements in space and a mass- and energy conserving variant of the Crank-Nicolson method due to Sanz-Serna in time. The novel aspects of the analysis are the incorporation of rough, discontinuous potentials in the context of weak and strong disorder, the consideration of some general class of non-linearities, and the proof of convergence with rates in $L^{\infty}(L^2)$ under moderate regularity assumptions that are compatible with discontinuous potentials. For sufficiently smooth potentials, the rates are optimal without any coupling condition between the time step size and the spatial mesh width.
\end{abstract}

\section{Introduction}
This paper is devoted to nonlinear Schr\"odinger equations (NLS) of the form
$$
\ci \partial_t u = - \triangle u + V u + \gamma( |u|^2 ) u 
$$
Here, $u(x,t)$ is a complex valued function, $V(x)$ is a possibly rough/discontinuous potential and $\gamma : [0,\infty) \rightarrow [0,\infty)$ is a smooth function (in terms of the density $|u|^2$) that describes the nonlinearity. A common example is the cubic nonlinearity given by $\gamma(|u|^2)u=\beta |u|^2u$, for $\beta \in \R$, for which the equation is known as the Gross-Pitaevskii equation modeling for instance the dynamics of Bose-Einstein condensates in a potential trap \cite{Gro61,LSY01,Pit61}. In this paper we study Galerkin approximations of the NLS using a finite element space discretization to account for missing regularity due to a possibly discontinuous potential and we use a Crank-Nicolson time discretization to conserve two important invariants of the NLS, namely the mass and the energy. We aim at deriving rate-explicit a priori error estimates and the influence of rough potentials on these rates.

The list of references to numerical approaches for solving the NLS (both time-dependent and stationary) is long and includes \cite{ABB13d,AnD14,BaC13b,BaD04,BaT03,CCM10,DaH10,DaK10,Gau11,HMP14b,JKM14,Lub08,ThA12,Tha12b} and the references therein. For software libraries allowing the simulation of the time-dependent Gross-Pitaevskii equation we refer to \cite{AnD15,SSB16,VVB12}. A priori error estimates for $P1$ finite element approximations for the NLS have been studied in \cite{ADK91,KaM98,KaM99,San84,Tou91,Wan14,Zou01,HeM15}, where an implicit Euler discretization is considered in \cite{ADK91,HeM15}, a mass conservative one-stage Gauss-Legendre implicit Runge-Kutta scheme is analyzed in \cite{Tou91,HeM15}, mass conservative linearly implicit two-step finite element methods are treated in \cite{Zou01,Wan14} and higher order (DG and CG) time-discretizations are considered in \cite{KaM98,KaM99} (however these higher order schemes can lack conservation properties). The only scheme that is both mass and energy conservative at the same time is the modified Crank-Nicolson scheme analyzed by Sanz-Serna \cite{San84} and Akrivis et al. \cite{ADK91}, which is also the approach that we shall follow in this contribution.

The analysis of this modified Crank-Nicolson scheme is devoted to optimal $L^2$-error estimates for sufficiently smooth solutions in both classical papers \cite{San84} and \cite{ADK91}. Sanz-Serna treats the one-dimensional case $d=1$ and periodic boundary conditions and Akrivis et al. consider $d=1,2,3$ and homogeneous Dirichlet boundary conditions. Although the modified Crank-Nicolson scheme is implicit, in both works, optimal error estimates require a constraint on the coupling between the time step $\deltat{}$ and the mesh size $h$. In \cite{San84}, the constraint reads $\deltat{}\lesssim h$ whereas a relaxed constraint of the form $\deltat{} \lesssim h^{d/4}$ is required in \cite{ADK91}. The results are related to the case of the earlier mentioned cubic nonlinearity of the form $\gamma(|u|^2)u=\beta |u|^2u$ and a potential is not taken into account. Finally, we also mention the results obtained by Bao and Cai \cite{BaC12,BaC13} in the context of a finite difference discretization in space. Here, similar coupling conditions are obtained as by Sanz-Serna.

The present paper generalizes the results of Akrivis et al. \cite{ADK91} 
to the case of a broader class of nonlinearities and, more importantly, accounts for potential terms in the NLS. If the potential is sufficiently smooth, even the previous constraints on the time step can be removed without affecting the optimal convergence rates.
To the best of our knowledge, the only other paper that includes potential terms in a finite element based NLS discretization is  \cite{HeM15} which uses a one-stage Gauss-Legendre implicit Runge-Kutta scheme that is not energy-conserving.

While the aforementioned results essentially require continuous potentials, many physically relevant potentials are discontinuous and very rough. Typical examples are disorder potentials \cite{NBP13} or potentials representing quantum arrays in the context Josephson oscillations \cite{WWW98,ZSL98}. 
As the main result of the paper, we will also prove convergence in the presence of such potentials with convergence rates. The rates are lower than the optimal ones for smooth solutions and a coupling condition between the discretization parameters shows up again. Note, however, that this new coupling condition is very different from the one mentioned above as it forces the spatial mesh size to be sufficiently small depending on the time step. While the sharpness of these results for rough potentials remains open, we shall stress that we are not aware of a proof of convergence of any discretization (finite elements, finite differences, spectral methods, etc.) of the NLS in the presence of purely $L^{\infty}$-potentials and that we close this gap with this paper. We note again that we decided for the use of a finite element space discretization as it allows us to work in very low regularity regimes that cannot be handled with spectral or finite difference approaches.
\medskip

The structure of this article is as follows.
Section~\ref{sec:problem} introduces the model problem and its discretization. The main results and the underlying assumptions are stated in Section \ref{sec:main-results}. 
Sections~\ref{s:errorsemi}--\ref{s:errorfull} are devoted to the proof of these results. We present numerical results in Section~\ref{sec:numexp}.
Some supplementary material regarding the feasibility of our assumptions is provided as Appendix~\ref{appendix-b}.

\section{Problem formulation and discretization}\label{sec:problem}

Let ${\mathcal{D}} \subset \R^d$ (for $d=2,3$) be a convex bounded polyhedron that defines the computational domain. We consider a real-valued nonnegative disorder potential $V \in L^{\infty}(\mathcal{D};\R)$. Besides being bounded, $V$ can be arbitrarily rough. Given such $V$, some finite time $T>0$ and some initial data $u^0 \in H^1_0(\mathcal{D}):=H^1_0(\mathcal{D},\C)$, we seek a wave function 
$u \in L^{\infty}([0,T],H^1_0({\mathcal{D}}))$ with $\partial_t  u \in L^{\infty}([0,T],H^{-1}({\mathcal{D}}))$
such that $u(\cdot,0)=u^0 $ and
\begin{eqnarray}
\label{model-problem}\ci \langle \partial_t u(\cdot,t) , w \rangle_{H^{-1},H^1_0}=
\Ltwo{\nabla u(\cdot,t)}{\nabla w}
+ \Ltwo{V \hspace{2pt} u(\cdot,t)}{ w}
+ \Ltwo{(\gamma( |u(\cdot,t)|^2 ) u(\cdot,t)}{w}
\end{eqnarray}
for all $w \in H^1_0({\mathcal{D}})$ and almost every $t \in (0,T]$. Note that any such solution automatically fulfills $u \in C^0([0,T],L^2({\mathcal{D}}))$ so that $u(\cdot,0)=u^0$ makes sense. The nonlinearity in the problem is described by a smooth (real-valued) function $\gamma : [0,\infty) \rightarrow [0,\infty)$ with $\gamma(0)=0$ and the growth condition
\begin{align*} 
|\gamma(|v|^2)v - \gamma(|w|^2)w| \le L(K) |v - w| \qquad \mbox{for all } v,w \in \C \mbox{ with } |v|,|w|\le K
\end{align*}
and
\begin{align*} 
0 \le L(s) \le C s^q \qquad \mbox{for } s \ge 0 \qquad \mbox{and } 
\begin{cases}
q \in [0,\infty) &\mbox{for } d=2,\\
q \in [0,2) &\mbox{for } d=3.
\end{cases}
\end{align*}
Observe that this implies by Sobolev embeddings that 
$\Ltwo{(\gamma( |v|^2 ) v}{w}$ is finite 
for any $v,w \in \V$. We define
$$
\Gamma(\rho):= \int_0^{\rho} \gamma(t) \hspace{2pt} dt \ge 0.
$$
Then, for any  $v \in \V$, the (non-negative) energy is given by
\begin{align*}
E(v) = \int_{\mathcal{D}} |\nabla v|^2 + \int_{\mathcal{D}} V \hspace{2pt} |v|^2  + \int_{\mathcal{D}} \Gamma(|v|^2).
\end{align*}

\begin{remark}[Existence]
\label{prop-exist-and-unique}
There exists at least one solution to problem \eqref{model-problem}. For a corresponding result we refer to \cite[Proposition 3.2.5, Remark 3.2.7, Theorem 3.3.5 and Corollary 3.4.2]{Caz03}. However, uniqueness is only known in exceptional cases. For instance, if $d\le 2$ and $q\le 2$ the solution is also unique (cf. \cite[Theorem 3.6.1]{Caz03}). For further settings that guarantee uniqueness, see \cite[Corollary 3.6.2, Remark 3.6.3 and Remark 3.6.4]{Caz03}.
\end{remark}
$\\$
{\it Temporal discretization.} We consider a time interval $[0,T]$ and a corresponding family of admissible partitions. A partition $\{ I_n| \hspace{2pt} n \in \N,  \hspace{4pt} 1 \le n \le N \}$ is admissible if the $n$'th time interval is given by $I_n := (t_{n-1},t_{n}]$ and if $0=t_0<t_1< \cdots < t_N = T$. Furthermore, we assume that the family of partitions is quasi-uniform, i.e. if $\deltat{n} := t_{n} - t_{n-1}$ denotes the $n$'th time step size and if the maximum is denoted by $\deltat{}:=\max_{1\le n \le N} \deltat{n}$, then there exists a (discretization independent) constant $c_q>0$ such that $\deltat{} \le c_q \min_{1\le n \le N} \deltat{n}$ for all partitions from the family.

$\\$
{\it Spatial discretization.} For the space discretization we consider a finite dimensional subspace $\Sh$ of $H^1_0(\mathcal{D})$ that is parametrized by a mesh size parameter $h$. We make two basic assumptions on $\Sh$ which are fulfilled for $P1$ Lagrange finite elements on quasi-uniform meshes. Let us for this purpose introduce the Ritz-projection $P_h : H^1_0({\mathcal{D}}) \rightarrow \Sh$. For $v \in H^1_0({\mathcal{D}})$ the Ritz-projection $P_h(v) \in S_h$ is the unique solution to the problem
\begin{align}
\label{definition-ritz-projection}\Ltwo{ \nabla v - \nabla P_h(v)) }{ \nabla w_h } =0 \qquad \mbox{for all } w_h\in S_h.
\end{align}
In the following, we make an assumption on the approximation quality of $\PR$, that is that there exists a generic $h$-independent constant $C_{\PR}$ such that
\begin{align}
\label{L2-estimate-PR} \| v - \PR(v) \|_{L^2(\mathcal{D})} \le C_{\PR} 
h^2 | v |_{H^2(\mathcal{D})} \qquad \mbox{for all } v \in H^1_0(\mathcal{D}) \cap H^2(\mathcal{D}).
\end{align}
The second assumption is the availability of a global inverse estimate, i.e. we assume that there exists a generic $h$-independent constant $C_{\mbox{\rm\tiny inv}}$ such that
\begin{align}
\label{inverse-estimate-PR}\| v_h \|_{L^{\infty}(\mathcal{D})} \le C_{\mbox{\rm\tiny inv}} h^{-d/2} \| v_h \|_{L^{2}(\mathcal{D})} 
\qquad \mbox{for all } v_h \in \Sh.
\end{align}
In addition, we assume the existence of  $h$-independent $C_{L^{\infty}}>0$ with
\begin{align}
\label{Linfty-stability-PR}\| \PR(v  )  \|_{L^{\infty}(\mathcal{D})} \le C_{L^{\infty}} \| v \|_{H^{2}(\mathcal{D})} \qquad \mbox{for all } v \in H^1_0(\mathcal{D}) \cap H^2(\mathcal{D}).
\end{align}
The above assumptions are standard in the context of finite elements if quasi-uniformity is available. For instance, for simplicial $P1$ Lagrange finite elements on a quasi-uniform mesh, the estimates \eqref{L2-estimate-PR} and \eqref{inverse-estimate-PR} are satisfied. The last property can be verified by splitting $\PR(v  ) = I_h(v) + (\PR(v)-I_h(v))$ for some $L^{\infty}$-stable Cl\'ement-type quasi-interpolation operator. The estimate \eqref{Linfty-stability-PR} then follows from inverse inequalities and standard $H^1$-estimates for $\PR$.

\bigskip
With these definitions, we introduce the fully discrete Crank-Nicolson method as follows.

\begin{definition}[Fully discrete Crank-Nicolson Method for NLS]
\label{crank-nic-gpe}
We consider the space and time discretizations as detailed above.
Let $u^0=u(\cdot,0)$ be the initial value from problem \eqref{model-problem} and let $\hatuh^0:=\PR(u^0)\in\Sh$.
Then for $n \ge 1$, the fully discrete Crank-Nicolson approximation $\hatuh^{n} \in \Sh$ is given by
\begin{eqnarray}
\label{cnd-problem}\nonumber\lefteqn{\Ltwo{ \hatuh^{n}}{v} +
\deltat{n} \hspace{2pt}\ci \hspace{2pt}
\Ltwo{\nabla \hatuh^{n-\frac{1}{2}}}{\nabla v}
+ \deltat{n} \hspace{2pt}\ci \hspace{2pt}
\Ltwo{V \hspace{2pt} \hatuh^{n-\frac{1}{2}}}{ v}}\\
&+& \deltat{n} \hspace{2pt}\ci \hspace{2pt} \Ltwo{ \frac{\Gamma(|\hatuh^{n}|^2) - \Gamma(|\hatuh^{n-1}|^2)}{|\hatuh^{n}|^2 - |\hatuh^{n-1}|^2} \hatuh^{n-\frac{1}{2}}}{v} = \Ltwo{ \hatuh^{n-1}}{v} 
\end{eqnarray}
for all $v \in \Sh$ and where $\hatuh^{n-\frac{1}{2}}:=(\hatuh^{n}+\hatuh^{n-1})/2$.
\end{definition}

The scheme is mass conserving and energy conserving, i.e. we have
\begin{align*}
\| \hatuh^{n} \|_{L^2(\mathcal{D})} = \| \hatuh^{0} \|_{L^2(\mathcal{D})} \qquad \mbox{and} \qquad
E(\hatuh^{n}) = E(\PR(u^0))
\end{align*}
for all $0\le n \le N$. The mass conservation is verified by testing with $v=\hatuh^{n-\frac{1}{2}}$ in \eqref{cnd-problem} and taking the real part. The energy conservation is verified by testing in \eqref{cnd-problem} with $v=\hatuh^{n}-\hatuh^{n-1}$ and taking the imaginary part.

The conservation properties do not immediately guarantee robustness with respect to numerical perturbations (for instance arising from round-off errors), however, it can be proved that even the perturbed approximations remain uniformly bounded.
\begin{lemma}[Stability under numerical perturbation]\label{lem:stab}
Let $N>1$ and let $F^n \in L^2(\mathcal{D})$ (for $1\le n \le N$) be an $L^2$-perturbation of the discrete problem. We can think of $F^n$ as representing numerical errors. Let $\hatuh^{n} \in \Sh$ for $n\ge1$ be any solution to the (fully-discrete) perturbed problem
\begin{eqnarray*}
\lefteqn{\Ltwo{ \hatuh^{n} - \hatuh^{n-1}}{v} +
\deltat{n} \hspace{2pt}\ci \hspace{2pt}
\Ltwo{\nabla \hatuh^{n-\frac{1}{2}}}{\nabla v}
+ \deltat{n} \hspace{2pt}\ci \hspace{2pt}
\Ltwo{V \hspace{2pt} \hatuh^{n-\frac{1}{2}}}{ v}}\\
&+& \deltat{n} \hspace{2pt}\ci \hspace{2pt} \biggl\langle \frac{\Gamma(|\hatuh^{n}|^2) - \Gamma(|\hatuh^{n-1}|^2)}{|\hatuh^{n}|^2 - |\hatuh^{n-1}|^2} \hatuh^{n-\frac{1}{2}},v\biggr\rangle_{L^2(\mathcal{D})} =  \deltat{n} \hspace{2pt}\ci \hspace{2pt} \Ltwo{ F^{n} }{ v }
\end{eqnarray*}
for all $v \in \Sh$. Then the solutions remain uniformly bounded in $L^2$ with
\begin{eqnarray*}
\| \hatuh^{N} \|_{L^2(\mathcal{D})}^2 
&\le& e^{4} \left( \| P_h(u^0) \|_{L^2(\mathcal{D})}^2 + T\sum_{n=1}^N \deltat{n} \| F^{n} \|_{L^2(\mathcal{D})}^2  \right).
\end{eqnarray*}
\end{lemma}
\begin{proof}
We test in the problem formulation with $v=\hatuh^{n} + \hatuh^{n-1}$ and take the real part. This yields
\begin{eqnarray*}
\lefteqn{\| \hatuh^{n} \|_{L^2(\mathcal{D})}^2 - \| \hatuh^{n-1} \|_{L^2(\mathcal{D})}^2
= - \deltat{n} \hspace{2pt} \Im \Ltwo{ F^{n} }{ \hatuh^{n} + \hatuh^{n-1} }}\\
&\le& \frac{T}{2}\deltat{n} \| F^{n} \|_{L^2(\mathcal{D})}^2 + \frac{\deltat{n}}{T} \| \hatuh^{n} \|_{L^2(\mathcal{D})}^2 + \frac{\deltat{n}}{T} \| \hatuh^{n-1} \|_{L^2(\mathcal{D})}^2.
\end{eqnarray*}
Hence
\begin{eqnarray*}
\| \hatuh^{n} \|_{L^2(\mathcal{D})}^2 \le \left( 1 + \frac{2 \deltat{n}}{T-\deltat{n}} \right) \| \hatuh^{n-1} \|_{L^2(\mathcal{D})}^2
+ \frac{1}{2}\frac{\deltat{n}}{(T-\deltat{n})} T^2 \| F^{n} \|_{L^2(\mathcal{D})}^2.
\end{eqnarray*}
Applying this iteratively gives us
\begin{eqnarray*}
\| \hatuh^{N} \|_{L^2(\mathcal{D})}^2 &\le& e^{ \sum_{n=1}^N \frac{2 \deltat{n}}{T-\deltat{n}}} \left( \| \hatuh^{0} \|_{L^2(\mathcal{D})}^2
+ \sum_{n=1}^N \frac{1}{2}\frac{\deltat{n}}{(T-\deltat{n})} T^2 \| F^{n} \|_{L^2(\mathcal{D})}^2  \right)\\
&\le& e^{4} \left( \| \hatuh^{0} \|_{L^2(\mathcal{D})}^2
+ T \sum_{n=1}^N \deltat{n} \| F^{n} \|_{L^2(\mathcal{D})}^2  \right).
\end{eqnarray*}
\end{proof}

\section{Main results}\label{sec:main-results}

While the basic stability of the method in Lemma~\ref{lem:stab} does not require any additional smoothness assumptions, our quantified convergence and error analysis of the method relies on the regularity of $u$. We will use three types of regularity assumptions. 
\begin{itemize}
\item[(R1)] Assume that $u \in C^{0}( [0,T],H^2(\mathcal{D}))$, $\partial_t u \in L^4(\mathcal{D} \times (0,T))$ and 
$\partial_{tt} u \in L^2(\mathcal{D} \times (0,T))$.
\item[(R2)] Assume that $u^0 \in H^2(\mathcal{D})$ and $\partial_t u \in L^2(0,T;H^2(\mathcal{D}))$.
\item[(R3)] Assume that $\partial_{tt} u \in L^2(0,T;H^2(\mathcal{D}))$.
\end{itemize}
The first assumption allows the proof of convergence rates for the time-discretization and the second one is related to the optimal convergence rates for the space-discretization. 
Note that the high spatial regularity $u \in C^{0}( [0,T], H^2(\mathcal{D}))$ in (R1) implies that $u \in L^{\infty}(\mathcal{D} \times (0,T))$ for $d\le 3$ which is crucial for our proofs as they rely on uniform $L^{\infty}$-bounds for the discrete solutions and there is no hope for such thing if the continuous solution is unbounded in $L^{\infty}(\mathcal{D} \times (0,T))$. 
The third assumption (R3) will be used to obtain  optimal convergence rates for the time-discretization in the case of smooth potentials. We cannot expect (R3) to hold in the case of rough disorder potentials $V \in L^{\infty}(\mathcal{D})$. It is, however, possible to show that the assumptions (R1) and (R2) do not conflict with disorder potentials. We discuss this aspect in more detail in Appendix~\ref{appendix-b}. 

Before we state the main results, we shall show that every smooth solution that satisfies (R1) must be unique. Recall that we cannot guarantee uniqueness in general (cf. Remark~\ref{prop-exist-and-unique}). 
\begin{lemma}[Uniqueness of smooth solutions]
Any two solutions of the NLS \eqref{model-problem} that fulfill (R1) must be identical.
\end{lemma}
\begin{proof}
Let $g(v):=V v + \gamma(|v|^2)v$ for $v\in H^1_0(\Omega)$ and 
let $u_1$ and $u_2$ denote two smooth solutions in the sense that $u^{(1)},u^{(2)} \in C^{0}( [0,T],H^2(\mathcal{D}))$. By Sobolev embedding we can define
$K:= \max_{k=1,2} \| u^{(k)} \|_{L^{\infty}( [0,T] \times \mathcal{D})}< \infty$.
With
$e:=u^{(1)}- u^{(2)}$ we obtain for $t\ge 0$
\begin{align*}
\frac{1}{2} \frac{d}{dt} \| e(t) \|_{L^2(\mathcal{D})}^2 = \Im \int_{\mathcal{D}} (g(u^{(1)}(t)) - g(u^{(1)}(t))) \overline{e(t)}
\le \left( L(K)+\|V\|_{L^{\infty}(\mathcal{D})} \right) \| e(t) \|_{L^2(\mathcal{D})}^2.
\end{align*}
Time integration and $e(0)=0$ then yield
\begin{align*}
\| e(t) \|_{L^2(\mathcal{D})}^2 \le \left( L(K)+\|V\|_{L^{\infty}(\mathcal{D})} \right) \int_0^t \| e(s) \|_{L^2(\mathcal{D})}^2 \hspace{2pt} ds.
\end{align*}
Hence, Gr\"onwall's inequality can be applied and shows $\| e(t) \|_{L^2(\mathcal{D})}^2=0$ for all $t$.
\end{proof}

The first main result of this paper states that, under the assumption of sufficient regularity,
the Crank-Nicolson scheme \eqref{cnd-problem} admits a solution that remains uniformly bounded in $L^{\infty}$ and we obtain optimal convergence rates for the $L^{\infty}(L^2)$-error, independent of the coupling between the mesh size $h$ and the time-step size $\deltat{}$.

\begin{theorem}[Estimates for smooth potentials]
\label{main-theorem-2-a}
Under the regularity assumption (R1), (R2) and (R3), there exist positive constants $\hat{h}>0$ and $\hat{\deltat{}}>0$ such that for all partitions with parameters  $\deltat{}<\hat{\deltat{}}$ and $h < \hat{h}$ 
there exists a unique solution $\hatuh^{n} \in \Sh$ to the fully discrete Crank-Nicolson scheme \eqref{cnd-problem} with
\begin{align*}
\sup_{0\le n \le N} \| \hatuh^{n} \|_{L^{\infty}(\mathcal{D})} \le M,
\end{align*}
where $M:=1+ 2 \| u \|_{L^{\infty}(\mathcal{D} \times (0,T))} 
+ \sup_{0\le n \le N} \| u^n \|_{H^{2}(\mathcal{D})}$ and $u^{n}:=u(\cdot,t_n)$. 
Moreover, the a priori error estimate
\begin{align*}
\sup_{0\le n \le N} \| \hatuh^{n} - u^{n} \|_{L^2(\mathcal{D})} \le C \left( h^2 +\deltat{}^2 \right)
\end{align*}
holds with some constant $C>0$ that may depend on $u$, $\gamma$, $V$, $\mathcal{D}$ and the constants appearing in \eqref{L2-estimate-PR}-\eqref{Linfty-stability-PR} but not on the mesh parameters $\deltat{}$ and $h$.
\end{theorem}
The uniqueness of fully discrete approximations in Theorem~\ref{main-theorem-2-a} is to be understood in the sense that any other family of approximations must necessarily diverge in $L^{\infty}$ as $\deltat{},h\rightarrow 0$. The second main result applies to the case of rough potentials. 
\begin{theorem}[Estimates for disorder potentials]
\label{main-theorem-2-b}
Assume only (R1) and (R2). Then there exists $\hat{\deltat{}}>0$ such that for all partitions with parameters $\deltat{}<\hat{\deltat{}}$ and $h^{4-d-\alpha} \lesssim \deltat{}^2$ for some $\alpha>0$ there exists a unique solution $\hatuh^{n} \in \Sh$ to the fully-discrete Crank-Nicolson scheme \eqref{cnd-problem} such that
\begin{align*}
\sup_{0\le n \le N} \| \hatuh^{n} \|_{L^{\infty}(\mathcal{D})} \le M,
\end{align*}
with $M$ as defined in Theorem~\ref{main-theorem-2-a}, and the a priori error estimate 
\begin{align*}
\sup_{0\le n \le N} \| \hatuh^{n} - u^{n} \|_{L^2(\mathcal{D})} \le C \left( h^{(d+\alpha)/2} + \deltat{} \right)
\end{align*}
holds for some constant $C=C(u,\gamma,V,\mathcal{D},\PR,\alpha)$ independent of $h$ and $\deltat{}$. 
\end{theorem}

Sections~\ref{s:errorsemi}--\ref{s:errorfull} below are devoted to the proof of Theorems \ref{main-theorem-2-a} and \ref{main-theorem-2-b}.

\begin{remark}[Coupling constraint]
The results of Theorem \ref{main-theorem-2-b} are valid under the constraint $h^{4-d-\alpha} \lesssim \deltat{}^2$ for some $\alpha>0$. This means that the mesh size needs to be small enough compared the time step size. Observe that this is a rather natural assumption if the potential $V$ is indeed a rough potential (as addressed in the theorem). In such a case we wish use a fine spatial mesh to resolve the variations of $V$, whereas the time step size is comparably large. Hence, the constraint is not critical. Conversely, the constraints appearing in works by Sanz-Serna \cite{San84} and Akrivis et al. \cite{ADK91} are of a completely different nature, as they require the time step size to be small compared to the mesh size. Therefore, using a fine spatial mesh to resolve the structure of $V$ would impose small time steps as well. 
\end{remark}

\section{Error analysis for the semi-discrete method}\label{s:errorsemi}
In this section we shall consider a semi-discrete Crank-Nicolson approximation given as follows.

\begin{definition}[Semi-discrete Crank-Nicolson Method for NLS]
\label{semi-discrete-crank-nic-gpe}
Let $u^0=u(\cdot,0)$ be the initial value from problem \eqref{model-problem} and let $\hatu^0:=u^0$. Then for $n \ge 1$, we define the semi-discrete Crank-Nicolson approximation $\hatu^{n} \in H^1_0(\mathcal{D})$ as the solution to
\begin{eqnarray}
\label{semi-disc-cnd-problem}\nonumber\lefteqn{\Ltwo{ \hatu^{n}}{v} +
\deltat{n} \hspace{2pt}\ci \hspace{2pt}
\Ltwo{\nabla \hatu^{n-\frac{1}{2}}}{\nabla v}
+ 
\deltat{n} \hspace{2pt}\ci \hspace{2pt}
\Ltwo{V \hspace{2pt} \hatu^{n-\frac{1}{2}}}{ v}
 }\\
&\enspace& + \deltat{n} \hspace{2pt}\ci \hspace{2pt} \Ltwobig{ \frac{\Gamma(|\hatu^{n}|^2) - \Gamma(|\hatu^{n-1}|^2)}{|\hatu^{n}|^2 - |\hatu^{n-1}|^2} \hatu^{n-\frac{1}{2}}}{v} = \Ltwo{ \hatu^{n-1}}{v} 
\end{eqnarray}
for all $v \in H^1_0(\mathcal{D})$ and where $\hatu^{n-\frac{1}{2}}:=(\hatu^{n}+\hatu^{n-1})/2$.
\end{definition}

We want to prove that the above problem is well-posed and we want to estimate the $L^2$- and $H^2$-error between $\hatu$ and the exact solution $u$. This requires some auxiliary results that allow us to control the error arising from the nonlinearity.

\subsection{A truncated auxiliary problem}
 
We start with introducing
a truncated version of the (possibly) nonlinear function $\gamma$. With this truncated function, we introduce an auxiliary problem that is central for our analysis. 
 
\begin{lemma}\label{truncation-lemma}
Let $M \in \R$ be a constant with $M \ge \| u \|_{L^{\infty}(\mathcal{D} \times (0,T))}$. Then there exists a smooth function $\gamma_{M} : [0,\infty) \rightarrow [0,\infty)$ and generic constants $C>0$ such that for all $z \in \C$ and all $k\in \{ 0,1,2 \}$
\begin{align}
\label{f_M_cond_1}\gamma_M(|z|^2) &= \gamma( |z|^2 )\hspace{10pt}\mbox{if } |z|\le M; \\
\label{f_M_cond_3a}\| \gamma_M^{(k)} \|_{L^{\infty}(0,\infty)} &\le C \gamma^{M,k}, \qquad \mbox{where } \gamma^{M,k}:=\| \gamma^{(k)} \|_{L^{\infty}(0,M^2)}.
\end{align}
Furthermore, for the antiderivative $\Gamma_{M}(s)=\int_0^s \gamma_{M}(t) \hspace{2pt}dt$
it holds for all $v_1,v_2,w_1,w_2 \in \C$ with $|v_1|,|w_1|\le M$:
\begin{eqnarray}
\label{f_M_cond_3}
\lefteqn{\left| \left( \frac{\Gamma_{M}(|v_1|^2) - \Gamma_{M}(|w_1|^2)}{|v_1|^2 - |w_1|^2} - \frac{\Gamma_{M}(|v_2|^2) - \Gamma_{M}(|w_2|^2)}{|v_2|^2 - |w_2|^2} \right) (v_1 + w_1 ) \right|}\\
\nonumber &\le& C  \left(\sum_{k=1}^2 M^{2k-1} \gamma^{M,k}\right) |v_1-w_1|^2
 + C 
  \left(\sum_{k=0}^2 M^{2k} \gamma^{M,k}\right) ( |v_1-v_2| + |w_1-w_2| ).
\end{eqnarray}
\end{lemma}

Before we can prove Lemma \ref{truncation-lemma} we need to introduce an inequality that we will frequently use in the rest of the paper.
 
\begin{lemma}
\label{lemma01ap}Let $\GammaF : [0,\infty) \rightarrow [0,\infty)$ be a three times continuously differentiable function with locally bounded derivatives.
Then, for every $z_0, z_1 \in \C$ with (w.l.g.) $|z_0|^2 \le |z_1|^2$ it holds
\begin{eqnarray*}
\lefteqn{\left| \frac{\GammaF(|z_0|^2) - \GammaF(|z_1|^2)}{|z_0|^2 - |z_1|^2} 
- \GammaF^{\prime}(\left| \frac{z_0 + z_1}{2} \right|^2)\right|}\\
&\le& |z_0-z_1|^2 \left( \frac{1}{4}  \| \GammaF^{\prime\prime}\|_{L^{\infty}(|z_0|^2,|z_1|^2)} + 
\frac{1}{3} \left| |z_0|+|z_1| \right|^2   \| \GammaF^{\prime\prime\prime}\|_{L^{\infty}(|z_0|^2,|z_1|^2)} \right).
\end{eqnarray*}
\end{lemma} 
\begin{proof}
Let us define $z_\half:=\frac{z_0+z_1}{2}$. First, we observe that
\begin{eqnarray*}
\lefteqn{\frac{\left( |z_0|^2 - |z_\half|^2\right)^2 - \left( |z_1|^2 - |z_\half|^2\right)^2}{|z_0|^2 - |z_1|^2}}\\
&=& \frac{\left( |z_0|^4 - 2 |z_0|^2 |z_\half|^2 + |z_\half|^4 \right) - \left( |z_1|^4 - 2 |z_1|^2 |z_\half|^2 + |z_\half|^4 \right)}{|z_0|^2 - |z_1|^2}\\
&=& \frac{\left( |z_0|^4 - |z_1|^4 \right) + 2 |z_\half|^2 \left( |z_1|^2 - |z_0|^2 \right)}{|z_0|^2 - |z_1|^2}\\
&=&  |z_0|^2 + |z_1|^2 - 2 |z_\half|^2 \\
&=& |z_0|^2 + |z_1|^2 - \frac{1}{2} \left( |z_0|^2 + |z_1|^2 + z_0 \overline{z_1} + z_1 \overline{z_0} \right)\\
&=& \half \left( |z_0|^2 + |z_1|^2 - z_0 \overline{z_1} - z_1 \overline{z_0}  \right) = \half |z_0-z_1|^2.
\end{eqnarray*}
Hence
\begin{align}
\label{estrgk}\frac{\left( |z_0|^2 - |z_\half|^2\right)^2 - \left( |z_1|^2 - |z_\half|^2\right)^2}{|z_0|^2 - |z_1|^2} = \half |z_0-z_1|^2.
\end{align}
With that and using Taylor expansion for suitable $\xi_0$, $\xi_1 \in \C$ with $|z_0|^2 \le |\xi_0|^2,|\xi_1|^2 \le |z_1|^2$ we observe
\begin{eqnarray*}
\lefteqn{\frac{\GammaF(|z_0|^2) - \GammaF(|z_1|^2)}{|z_0|^2 - |z_1|^2}}\\
&=& \frac{\sum_{k=0}^2 \frac{1}{k!} \left( (|z_0|^2 - |z_\half|^2)^k - (|z_1|^2 - |z_\half|^2)^k \right) \GammaF^{(k)}(|z_\half|^2)}{|z_0|^2 - |z_1|^2}\\
&\enspace& \qquad +\frac{1}{6}\frac{(|z_0|^2 - |z_\half|^2)^3 \GammaF^{\prime\prime\prime}(|\xi_0|^2) - (|z_1|^2 - |z_\half|^2)^3 \GammaF^{\prime\prime\prime}(|\xi_1|^2)}{|z_0|^2 - |z_1|^2}\\
&=&
\frac{ \left( (|z_0|^2 - |z_\half|^2) - (|z_1|^2 - |z_\half|^2) \right)}{|z_0|^2 - |z_1|^2}\GammaF^{\prime}(|z_\half|^2)\\
&\enspace& \qquad +
\frac{1}{2} \frac{ (|z_0|^2 - |z_\half|^2)^2 - (|z_1|^2 - |z_\half|^2)^2}{|z_0|^2 - |z_1|^2}\GammaF^{\prime\prime}(|z_\half|^2)\\
&\enspace& \qquad +\frac{1}{6}\frac{(|z_0|^2 - |z_\half|^2)^3 \GammaF^{\prime\prime\prime}(|\xi_0|^2) - (|z_1|^2 - |z_\half|^2)^3 \GammaF^{\prime\prime\prime}(|\xi_1|^2)}{|z_0|^2 - |z_1|^2}\\
&\overset{\eqref{estrgk}}{=}&
\GammaF^{\prime}(|z_\half|^2) +
\frac{1}{4} |z_0-z_1|^2 \GammaF^{\prime\prime}(|z_\half|^2)\\
&\enspace& \qquad +\frac{1}{6}\frac{(|z_0|^2 - |z_\half|^2)^3 \GammaF^{\prime\prime\prime}(|\xi_0|^2) - (|z_1|^2 - |z_\half|^2)^3 \GammaF^{\prime\prime\prime}(|\xi_1|^2)}{|z_0|^2 - |z_1|^2}.
\end{eqnarray*}
Since for $k=1,2$
\begin{align*}
\left| |z_\ell|^2 - |z_\half|^2\right| \le \left| |z_0|^2 - |z_1|^2\right|,
\end{align*}
we obtain
\begin{eqnarray*}
\lefteqn{\left| \frac{\GammaF(|z_0|^2) - \GammaF(|z_1|^2)}{|z_0|^2 - |z_1|^2} 
- \GammaF^{\prime}(|z_\half|^2)\right|}\\
&\le& \frac{1}{4} |z_0-z_1|^2 \| \GammaF^{\prime\prime\prime}\|_{L^{\infty}(|z_0|^2,|z_1|^2)} + 
\frac{1}{3} ||z_0|^2-|z_1|^2|^2   \| \GammaF^{\prime\prime\prime}\|_{L^{\infty}(|z_0|^2,|z_1|^2)}.
\end{eqnarray*}
\end{proof}

\begin{proof}[Proof of Lemma \ref{truncation-lemma}]
In the following, we let $C$ denote a generic constant.
Let us define $\theta:=\gamma(M^2)$ and let $\gamma_M : \R \rightarrow \R$ be a curve that fulfills $\gamma_M(s)=\gamma(s)$ for $s \le \theta$ and $\gamma_M(s)=2 \theta$ for $s \ge 2\theta$. By polynomial interpolation we can chose $\gamma_M$ in such a way that it is a polynomial on the interval $[\theta,2\theta]$ and such that globally $\gamma_M \in C^{5}(0,\infty)$. This proves \eqref{f_M_cond_1}. Since we have for $s \ge 2\theta$ that $\gamma_M(s)=2 \theta$ and $\gamma_M^{(k)}(s)=0$ for $1 \le k \le 5$, we conclude that there exists a constant $C$ that only depends on the polynomial degree chosen for the interpolation such that \eqref{f_M_cond_3a} is fulfilled.

Next, we want to prove (\ref{f_M_cond_3}) and recall that $|v_1|,|w_1|\le M$. Let for simplicity
\begin{align}
\label{def-of-Theta}\Theta_\ell :=  \frac{\Gamma_{M}(|v_\ell|^2) - \Gamma_{M}(|w_\ell|^2)}{|v_\ell|^2 - |w_\ell|^2}.
\end{align}
We write the left-hand side of (\ref{f_M_cond_3}) as
\begin{align*}
(v_1 + w_1) \Theta_1  -  (v_2 + w_2)\Theta_2 = (v_1 + w_1) (\Theta_1 - \Theta_2)  +(v_1 - v_2 + w_1 - w_2) \Theta_2 .
\end{align*}
Since $|\Gamma_{M}^{\prime}|=|\gamma_M|\le C \| \gamma \|_{L^{\infty}(0,M^2)}$ is bounded we immediately have that
$$
|(v_1 - v_2 + w_1 - w_2) \Theta_2 | \le C \| \gamma \|_{L^{\infty}(0,M^2)} ( |v_1 - v_2| + |w_1 - w_2| ).
$$
Hence, it remains to estimate the term $(v_1 + w_1)\Theta_1$ which we split into three parts.
\begin{eqnarray}
\label{error-splitting-for-f_M_cond_3}\lefteqn{ (v_1 + w_1 ) (\Theta_1- \Theta_2)}\\
\nonumber&=& 
\underset{=:\mbox{I}}{\underbrace{\left( \frac{\Gamma_{M}(|v_1|^2) - \Gamma_{M}(|w_1|^2)}{|v_1|^2 - |w_1|^2} 
- \Gamma_{M}^{\prime}\left(\left|\frac{v_1+w_1}{2}\right|^2\right)
\right) (v_1 + w_1 )}}\\
\nonumber&\enspace& + 
\underset{=:\mbox{II}}{\underbrace{\left( 
\Gamma_{M}^{\prime}\left(\left|\frac{v_1+w_1}{2}\right|^2\right)
- \Gamma_{M}^{\prime}\left(\left|\frac{v_2+w_2}{2}\right|^2\right)
\right) (v_1 + w_1 )}}\\
\nonumber&\enspace& + 
\underset{=:\mbox{III}}{\underbrace{\left( 
\Gamma_{M}^{\prime}\left(\left|\frac{v_2+w_2}{2}\right|^2\right)
- \frac{\Gamma_{M}(|v_2|^2) - \Gamma_{M}(|w_2|^2)}{|v_2|^2 - |w_2|^2}
\right) (v_1 + w_1 )}}.
\end{eqnarray}
We start with estimating term I. 
We use Lemma \ref{lemma01ap} with the boundedness of $|v_1|$ and $|w_1|$ to obtain
\begin{eqnarray*}
\mbox{I} = \left| \left( \frac{\Gamma_{M}(|v_1|^2) - \Gamma_{M}(|w_1|^2)}{|v_1|^2 - |w_1|^2} 
- \Gamma_{M}^{\prime}\left(\left|\frac{v_1+w_1}{2}\right|^2\right)
\right) (v_1 + w_1 ) \right| \le C (M\gamma^{M,1}+M^3 \gamma^{M,2}) |v_1-w_1|^2.
\end{eqnarray*}
For the second term we distinguish two cases. Case II.1: if $\left|\frac{v_2+w_2}{2}\right|^2 \ge 4 M^2$ we get
\begin{align}
\label{case.II.1}\left|\frac{v_1 + w_1}{2} \right| \le 2 M - M \le \left|\frac{v_2+w_2}{2}\right| - \left|\frac{v_1+w_1}{2}\right| 
\le \frac{1}{2} \left|v_1 - v_2\right| + \frac{1}{2} \left|w_1 - w_2\right|
\end{align}
and hence with the bound for $\Gamma_{M}^{\prime}$
\begin{eqnarray*}
\mbox{II} &=& \left| \left( 
\Gamma_{M}^{\prime}\left(\left|\frac{v_1+w_1}{2}\right|^2\right)
- \Gamma_{M}^{\prime}\left(\left|\frac{v_2+w_2}{2}\right|^2\right)
\right) (v_1 + w_1 ) \right| \\
&\le& C \gamma^{M,0} \left( \left|v_1 - v_2\right| + \left|w_1 - w_2\right| \right).
\end{eqnarray*}
Case II.2: if $\left|\frac{v_2+w_2}{2}\right|^2 \le 4 M^2$ we get with the Lipschitz-continuity of $\Gamma_{M}^{\prime}$ that
\begin{eqnarray*}
\mbox{II} &=& \left| \left( 
\Gamma_{M}^{\prime}\left(\left|\frac{v_1+w_1}{2}\right|^2\right)
- \Gamma_{M}^{\prime}\left(\left|\frac{v_2+w_2}{2}\right|^2\right)
\right) (v_1 + w_1 ) \right| \\
&\le& C \gamma^{M,1} M \left( \left| v_1 + w_1\right|^2
- \left| v_2 + w_2 \right|^2 \right) \\
&\le& C \gamma^{M,1} M^2 \left( \left| v_1 - v_2\right| + \left| w_1 - w_2 \right| \right).
\end{eqnarray*}

For term III, we distinguish again two cases. Case III.1: if $|\frac{v_2+w_2}{2}|^2 \ge 4 M^2$ we see
analogously to term Case II.1 that
\begin{eqnarray*}
\lefteqn{\mbox{III} = \left( 
\Gamma_{M}^{\prime}\left(\left|\frac{v_2+w_2}{2}\right|^2\right)
- \frac{\Gamma_{M}(|v_2|^2) - \Gamma_{M}(|w_2|^2)}{|v_2|^2 - |w_2|^2}
\right) (v_1 + w_1 )}\\
&\le& C \gamma^{M,0} |v_1 + w_1| \overset{\eqref{case.II.1}}{\le} C \gamma^{M,0} \left( \left|v_1 - v_2\right| +  \left|w_1 - w_2\right| \right).
\end{eqnarray*}
Case III.2: if $|\frac{v_2+w_2}{2}|^2 \le 4 M^2$ we can use Lemma \ref{lemma01ap} to obtain
\begin{eqnarray*}
\lefteqn{\mbox{III} = \left( 
\Gamma_{M}^{\prime}\left(\left|\frac{v_2+w_2}{2}\right|^2\right)
- \frac{\Gamma_{M}(|v_2|^2) - \Gamma_{M}(|w_2|^2)}{|v_2|^2 - |w_2|^2}
\right) (v_1 + w_1 )}\\
&\le& |v_2-w_2|^2 \left( \frac{1}{4}  \| \Gamma_M^{\prime\prime}\|_{L^{\infty}(0,\infty)} + 
\frac{1}{3} \left| |v_2|+w_2| \right|^2   \| \Gamma_M^{\prime\prime\prime}\|_{L^{\infty}(0,\infty)} \right) |v_1 + w_1|\\
&\le& C |v_2-w_2|^2 \left( M \gamma^{M,1} + 
M^3 \gamma^{M,2} \right).
\end{eqnarray*}
With the bounds for $|v_1|,|v_2|,|w_1|,|w_2|$ we get
\begin{align*}
 |v_2-w_2|^2 \lesssim  |v_2-v_1|^2 + |v_1-w_1|^2 + |w_1-w_2|^2
 \le  |v_1-w_1|^2 + C M  |v_1-v_2| + C M |w_1-w_2|.
\end{align*}
This shows for Case III.2 that
\begin{eqnarray*}
\mbox{III} &\le& C \left( M \gamma^{M,1} + 
M^3 \gamma^{M,2} \right) |v_1-w_1|^2 + C \left( M^2 \gamma^{M,1} + 
M^4 \gamma^{M,2} \right) ( |v_1-v_2| + |w_1-w_2| ).
\end{eqnarray*}
Combining the estimates proves (\ref{f_M_cond_3}).
\end{proof}

With the function $\gamma_M$ introduced in Lemma \ref{truncation-lemma} the {\it truncated} semi-discrete Crank-Nicolson approximation given as follows.

\begin{definition}[Semi-discrete Crank-Nicolson Method with truncation]
\label{truncated-semi-discrete-crank-nic-gpe}
Let $\Gamma_M$ be given as in Lemma \ref{truncation-lemma} and let $\hatuM{0}:=u^0$. Then for $n \ge 1$, we define the truncated semi-discrete Crank-Nicolson approximation $\hatuM{n} \in H^1_0(\mathcal{D})$ as the solution to
\begin{eqnarray}
\label{semi-disc-cnd-problem-truncation}\nonumber\lefteqn{\Ltwo{ \hatuM{n} }{v} +
\deltat{n} \hspace{2pt}\ci \hspace{2pt}
\Ltwo{\nabla \hatuM{n-\frac{1}{2}}}{\nabla v}
+ 
\deltat{n} \hspace{2pt}\ci \hspace{2pt}
\Ltwo{V \hspace{2pt} \hatuM{n-\frac{1}{2}}}{ v} }\\
&+& \deltat{n} \hspace{2pt}\ci \hspace{2pt} \Ltwobig{ \frac{\Gamma_M(|\hatuM{n}|^2) - \Gamma_M(|\hatuM{n-1}|^2)}{|\hatuM{n}|^2 - |\hatuM{n-1}|^2} \hatuM{n-\frac{1}{2}}}{v} = \Ltwo{ \hatuM{n-1}}{v}
\end{eqnarray}
for all $v \in H^1_0(\mathcal{D})$ and where $\hatuM{n-\frac{1}{2}}:=(\hatuM{n}+\hatuM{n-1})/2$.
\end{definition}

\begin{remark}
Since $M$ was chosen such that $\gamma_M(|u|^2)=\gamma(|u|^2)$ we have the identity
\begin{align}
\label{exact-solution-with-discrete-test-func}\Ltwo{ u(\cdot,t_n) }{ v } +\ci \int_{I_n} \Ltwo{ \nabla u }{ \nabla v } +\ci \int_{I_n} \Ltwo{ Vu + \gamma_M(|u|^2) u }{ v }
= \Ltwo{ u(\cdot,t_{n-1}) }{ v }.
\end{align}
for all $v\in H^1_0(\mathcal{D})$.
\end{remark}

\subsection{Existence of truncated approximations}

In order to investigate the properties of solutions to \eqref{semi-disc-cnd-problem-truncation}, we first need to show that there exists at least one solution. In order to show this, we recall the following conclusion from Brouwers fixed point theorem.
\begin{lemma}\label{brouwer-lemma}
Let $N \in \mathbb{N}$ and let $\overline{B_1(0)}:=\{ \boldsymbol{\alpha} \in \C^N| \hspace{2pt} |\boldsymbol{\alpha}|\le 1\}$ denote the closed unit disk in $\C^N$. Then every continuous function $g : \C^N \rightarrow \C^N$ with $\Re( g(\boldsymbol{\alpha}) \cdot \boldsymbol{\alpha} ) \ge 0$ for all $\boldsymbol{\alpha} \in \partial B_1(0)$ has a zero in $\overline{B_1(0)}$, i.e. a point $\boldsymbol{\alpha}_0 \in \overline{B_1(0)}$ with $g(\boldsymbol{\alpha}_0)=0$.
\end{lemma}
Second, we will also make use of following lemma that is a special case of the Vitali convergence theorem (cf. \cite[Theorem B.101]{Leo09})
\begin{lemma}\label{Vitali-convergence-theorem}
Recall that $\mathcal{D} \subset \R^d$ is bounded. A sequence $(g_k)_{k\in \mathbb{N}} \subset L^2(\mathcal{D})$ converges strongly to $g \in L^2(\mathcal{D})$ if and only if
\begin{enumerate}
\item $(g_k)_{k\in \mathbb{N}}$ converges to $g$ locally in measure and
\item $(g_k)_{k\in \mathbb{N}}$ is $2$-equi-integrable, i.e. for every $\eps>0$ there exists a $\delta_{\eps}>0$ such that $\| g_k \|_{L^2(S)}^2<\eps$ for all measurable subsets $S\subset \mathcal{D}$ with measure $\mu(S)<\delta_{\eps}$. 
\end{enumerate} 
\end{lemma}
This allows us to conclude that there exists at least one solution to problem \eqref{semi-disc-cnd-problem-truncation}. 
\begin{lemma}
\label{existence-of-semi-disc-trunc-solutions}
For every $n\ge1$ there exists at least one solution $\hatuM{n} \in H^1_0(\mathcal{D})$ to the truncated problem \eqref{semi-disc-cnd-problem-truncation}.
\end{lemma}

\begin{proof}
The proof is established in two steps. First, we show existence in finite dimensional subspaces, then we pass to the limit to establish existence for the infinite dimensional problem. For this purpose, let $\{ \phi_m \;|\; m \in \mathbb{N} \}$ denote a countable 
basis of $H^1_0(\mathcal{D})$. We define the finite-dimensional subspaces 
$$X_N:=\{ \phi_m | \hspace{2pt} 1 \le m \le N \}.
$$

{\it Step 1 - existence in $X_N$.} Let $\cdot$ denote the Euclidean inner product on $\mathbb{C}^N$
and let $\hatuM{n}$ denote a solution to \eqref{semi-disc-cnd-problem-truncation}.
For $n \ge 1$ we look for $z_N^n \in X_N$ with
\begin{eqnarray}
\label{semi-disc-cnd-problem-truncation-XN}\nonumber\lefteqn{\Ltwo{ z_N^n }{v} +
\deltat{n} \hspace{2pt}\ci \hspace{2pt}
\Ltwo{\nabla z_N^{n-\frac{1}{2}}}{\nabla v}
 + \deltat{n} \hspace{2pt}\ci \hspace{2pt} \biggl\langle \frac{\Gamma_M(|z_N^n|^2) -  \Gamma_M(| \hatuM{n-1} |^2)}{|z_N^{n}|^2 - | \hatuM{n-1} |^2} z_N^{n-\frac{1}{2}},v\biggr\rangle_{L^2(\mathcal D}}\\
&=& \Ltwo{ \hatuM{n-1} }{v} \hspace{340pt}
\end{eqnarray}
for all $v \in X_N$ and where $z_N^{n-\frac{1}{2}}:=(z_N^{n}+\hatuM{n-1})/2$.
We assume that $\hatuM{n-1} \in H^1_0(\mathcal{D})$ exists and want to show existence of $z_N^n \in X_N$.
In order to apply Lemma \ref{brouwer-lemma} we define $g : \C^{N} \rightarrow \C^{N}$ for $\boldsymbol{\alpha}\in \C^{N}$ by
\begin{eqnarray*}
\lefteqn{g_\ell(\boldsymbol{\alpha}) := - \deltat{n}^{-1} \hspace{2pt}\ci \sum_{m=1}^{N} \boldsymbol{\alpha}_m \Ltwo{ \phi_m }{ \phi_\ell } 
 +
\frac{1}{2} \sum_{m=1}^{N} \boldsymbol{\alpha}_m \hspace{2pt} 
\Ltwo{ \nabla \phi_m }{ \nabla \phi_\ell }
 +
\frac{1}{2} \sum_{m=1}^{N} \boldsymbol{\alpha}_m \hspace{2pt} 
\Ltwo{ V \phi_m }{  \phi_\ell }
}\\
&\enspace&
+ \frac{1}{2} \Ltwo{ \frac{\Gamma_M(\left| \sum_{m=1}^{N} \boldsymbol{\alpha}_m \phi_m  \right|^2) - \Gamma_M(|\hatuM{n-1}|^2)}{\left| \sum_{m=1}^{N} \boldsymbol{\alpha}_m \phi_m  \right|^2 - |\hatuM{n-1}|^2}  \left( \hatuM{n-1} + \sum_{m=1}^{N} \boldsymbol{\alpha}_m \phi_m
\right)
}{ \phi_\ell }
 +F_\ell,
\end{eqnarray*}
where $F\in \C^{N}$ is defined by
\begin{align*}
F_\ell := \frac{1}{2}\Ltwo{ \nabla \hatuM{n-1}}{ \nabla \phi_\ell } 
+ \frac{1}{2}\Ltwo{ V \hatuM{n-1}}{  \phi_\ell }
 + \ci \deltat{n}^{-1} \Ltwo{\hatuM{n-1}}{\lambda_\ell}.
\end{align*}
To verify existence of $\boldsymbol{\alpha}_0$ with $g(\boldsymbol{\alpha}_0)=0$, we need to show that there exists $K \in \R_{>0}$ such that $\Re( g(\boldsymbol{\alpha}) \cdot \boldsymbol{\alpha} ) \ge 0$ for all $\boldsymbol{\alpha} \in \C^{N}$ with $|\boldsymbol{\alpha}|= K$. We define $z_\alpha:=\sum_{m=1}^{N} \boldsymbol{\alpha}_m \phi_m$ and 
see that
\begin{eqnarray*}
\lefteqn{\Re( g(\boldsymbol{\alpha}) \cdot \boldsymbol{\alpha} ) =
\frac{1}{2}\| \nabla z_\alpha \|_{L^2({\mathcal{D}})}^2
+ \frac{1}{2} \Ltwo{ V z_\alpha}{ z_\alpha}
+ \Re( F \cdot \boldsymbol{\alpha} )} \\
&\enspace& \qquad + 
\frac{1}{2} \Re \Ltwo{ \frac{\Gamma_M(\left|z_\alpha \right|^2) - \Gamma_M(|\hatuM{n-1}|^2)}{\left| z_\alpha  \right|^2 - |\hatuM{n-1}|^2}  \left( \hatuM{n-1} +z_\alpha
\right)
}{ z_\alpha}\\
&\ge& \frac{1}{2}\| \nabla z_\alpha \|_{L^2({\mathcal{D}})}^2
- \left( \frac{1}{2} \| V \|_{L^{\infty}(\mathcal{D})} + \deltat{n}^{-1} \right)
 \|  \hatuM{n-1} \|_{L^2({\mathcal{D}})} \| z_\alpha \|_{L^2({\mathcal{D}})
}
- \frac{1}{2} \| \nabla \hatuM{n-1} \|_{L^2({\mathcal{D}})} \| \nabla  z_\alpha \|_{L^2({\mathcal{D}})}
\\
&\enspace& \qquad 
- \frac{1}{2} \Re \Ltwo{ \frac{\Gamma_M(\left|z_\alpha \right|^2) - \Gamma_M(|\hatuM{n-1}|^2)}{\left| z_\alpha  \right|^2 - |\hatuM{n-1}|^2}  \left( \hatuM{n-1} +z_\alpha
\right)
}{ \hatuM{n-1} }
\\
&\ge& \frac{1}{2}\| \nabla z_\alpha \|_{L^2({\mathcal{D}})} \left( \| \nabla z_\alpha \|_{L^2({\mathcal{D}})} 
- \sqrt{2} \hspace{2pt} \mbox{\rm diam}(\mathcal{D}) 
\left( \frac{1}{2} \| V \|_{L^{\infty}(\mathcal{D})} + \deltat{n}^{-1} \right)
\| \hatuM{n-1} \|_{L^2({\mathcal{D}})}
- \| \nabla \hatuM{n-1} \|_{L^2({\mathcal{D}})} \right)\\
&\enspace& \qquad 
- \frac{1}{2} \gamma^{M,0} \| \hatuM{n-1} +z_\alpha \|_{L^2(\mathcal{D})}
\| \hatuM{n-1}  \|_{L^2(\mathcal{D})}
\\
&\ge& \frac{1}{2} \| \nabla z_\alpha \|_{L^2({\mathcal{D}})} \left( \| \nabla z_\alpha \|_{L^2({\mathcal{D}})} - C_1 \right) - C_2,
\end{eqnarray*}
for some $\boldsymbol{\alpha}$-independent positive constants $C_1$ and $C_2$. Exploiting the equivalence of norms in finite dimensional Hilbert spaces we conclude the existence of (new) $\boldsymbol{\alpha}$-independent positive constants such that
$\Re( g(\boldsymbol{\alpha}) \cdot \boldsymbol{\alpha} ) \ge C_3 |\boldsymbol{\alpha}| \left( |\boldsymbol{\alpha}| - C_1 \right) - C_2$. Hence, for all sufficiently large $|\boldsymbol{\alpha}|$ we have $\Re( g(\boldsymbol{\alpha}) \cdot \boldsymbol{\alpha} ) \ge 0$ and therefore with Lemma \ref{brouwer-lemma} the existence of solutions to \eqref{semi-disc-cnd-problem-truncation-XN}, provided that $\hatuM{n-1}$ exists.

{\it Step 2 - existence in $H^1_0(\mathcal{D})$.} We proceed inductively to show the existence of $\hatuM{n}$ (the case $n=0$ with $\hatuM{0}=u^0$ is trivially fulfilled). Assume hence that $\hatuM{n-1} \in H^1_0(\mathcal{D})$ exists. Then we can apply {\it Step 1} to conclude that there exists $z_N^n \in X_N$ which is a solution to the finite dimensional problem \eqref{semi-disc-cnd-problem-truncation-XN}. It is easy to verify that problem \eqref{semi-disc-cnd-problem-truncation-XN} is energy conserving and hence
\begin{align*}
\| \nabla z_N^n \|_{L^2(\mathcal{D})}^2 \le \int_{\mathcal{D}} \left( |\nabla z_N^n|^2 + 
 V |z_N^n|^2
+
\Gamma(|z_N^n|^2) \right)
= E( z_N^n ) = E( \hatuM{n-1} ) = E( u^0 ) \le C_{\Gamma} \| u^0 \|_{H^1(\mathcal{D})}^2.
\end{align*}
Hence, for fixed $n$, the corresponding sequence of discrete solutions $(z_N^n)_{N \in \N} \subset H^1_0(\mathcal{D})$ is a bounded sequence in $H^1_0(\mathcal{D})$ with $\| \nabla z_N^n \|_{L^2(\mathcal{D})}^2 \le C_{\Gamma} \| u^0 \|_{H^1(\mathcal{D})}^2$. Consequently there exists a subsequence of $(z_N^n)_{N \in \N}$ (for simplicity again denoted by $(z_N^n)_{N \in \N}$) and a function $z^n_{\infty} \in H^1_0(\mathcal{D})$ such that
\begin{align*}
z_N^n \rightharpoonup z^n_{\infty} \mbox{ weakly in } H^1_0(\mathcal{D}) \qquad \mbox{and} \qquad 
z_N^n \rightarrow z^n_{\infty} \mbox{ strongly in } L^2(\mathcal{D}) \qquad \mbox{for } N \rightarrow \infty. 
\end{align*}
Here we used the Rellich embedding theorem. For arbitrary $v \in H^1_0(\mathcal{D})$ we see that
\begin{align*}
 \Ltwo{ z_N^n }{v} \overset{N \rightarrow \infty }\longrightarrow  \Ltwo{ z_{\infty}^n }{v}
 \quad \mbox{and}  \quad \Ltwo{\nabla z_N^{n-\frac{1}{2}}}{\nabla v}
  \overset{N \rightarrow \infty }\longrightarrow \frac{1}{2} \Ltwo{\nabla z_{\infty}^{n} + \hatuM{n-1} }{\nabla v}
\end{align*}
by the weak convergence in $H^1_0(\mathcal{D})$. It remains to investigate the term
\begin{eqnarray*}
A(z_N^n) := \frac{\Gamma_M(|z_N^n|^2) -  \Gamma_M(| \hatuM{n-1} |^2)}{|z_N^{n}|^2 - | \hatuM{n-1} |^2} \left( \hatuM{n-1} + z_N^{n} \right).
\end{eqnarray*}
We want to apply Vitali's theorem (Lemma \ref{Vitali-convergence-theorem}) to conclude that $A(z_N^n) \rightarrow A(z_{\infty}^n)$ strongly in $L^2(\mathcal{D})$. For that purpose, we need to verify convergence in measure and $2$-equi-integrability. To verify the first property, we exploit that $z_N^n\rightarrow z_{\infty}^n$ strongly in $L^2(\mathcal{D})$. Using the Tschebyscheff inequality we see that $z_N^n$ also converges to $z_{\infty}^n$ in measure. This implies in particular that from every subsequence of $(z_N^n)_{N\in \mathbf{N}}$ we can extract another subsequence such that $z_N^n$ converges to $z_{\infty}^n$ almost everywhere. On the other hand, by the continuity of $A$, the convergence almost everywhere is preserved when $A$ is applied to the sequence. Consequently, for every subsequence of $(A(z_{N}^n))_{N\in \mathbf{N}}$ one can extract another subsequence that converges a.e. to $A(z_{\infty}^n)$. This is equivalent to the property that $(A(z_{N}^n))_{N\in \mathbf{N}}$ converges locally in measure to $(A(z_{\infty}^n))_{N\in \mathbf{N}}$ (since $\mathcal{D}$ is bounded). Hence, we have the first requirement for Lemma \ref{Vitali-convergence-theorem}. For the second requirement we first observe that 
$|A(z_N^n)|\le C_1 |z_N^n| + C_2$ (with $C_1=  \gamma^{M,0}$ and $C_2(x)=\gamma^{M,0} |\hatuM{n-1}(x)|$). Hence, $|A(z_N^n)|$ is $2$-equi-integrable if $|z_N^n|$ is $2$-equi-integrable, which however follows immediately again from the strong convergence $z_N^n \rightarrow z_{\infty}^n$ in $L^2(\mathcal{D})$ and Lemma \ref{Vitali-convergence-theorem} (which works in both directions). In conclusion, Vitali's convergence  theorem applies and yields $A(z_N^n) \rightarrow A(z_{\infty}^n)$ strongly in $L^2(\mathcal{D})$ for $N\rightarrow \infty$. With this we have that there exists a subsequence of $(z_{N}^n)_{N\in \mathbf{N}}$ (still denoted by $z_{N}^n$) and $z^n_{\infty} \in H^1_0(\mathcal{D})$ such that 
\begin{align*}
z_N^n \rightharpoonup z^n_{\infty} \mbox{ weakly in } H^1_0(\mathcal{D}) \qquad \mbox{and} \qquad 
A(z_N^n) \rightarrow A(z^n_{\infty}) \mbox{ stronlgy in } L^2(\mathcal{D}) \qquad \mbox{for } N \rightarrow \infty. 
\end{align*}
With this we can pass to the limit in \eqref{semi-disc-cnd-problem-truncation-XN} to obtain that $z^n_{\infty} \in H^1_0(\mathcal{D})$ with $z_{\infty}^{n-\frac{1}{2}}:=(z^n_{\infty} + \hatuM{n-1})/2$ solves
\begin{eqnarray*}
\lefteqn{\Ltwo{ z^n_{\infty} }{v} +
\deltat{n} \hspace{2pt}\ci \hspace{2pt}
\Ltwo{ \nabla z_{\infty}^{n-\frac{1}{2}} }{\nabla v}
 + \deltat{n} \hspace{2pt}\ci \hspace{2pt} \Ltwo{ \frac{\Gamma_M(|z^n_{\infty}|^2) -  \Gamma_M(| \hatuM{n-1} |^2)}{|z^n_{\infty}|^2 - | \hatuM{n-1} |^2} z_{\infty}^{n-\frac{1}{2}}}{v}}\\
&=& \Ltwo{ \hatuM{n-1} }{v} \hspace{100pt} \mbox{for all } v\in H^1_0(\mathcal{D}).\hspace{200pt}
\end{eqnarray*}
Consequently, we showed iteratively the existence of a solution $\hatuM{n} \in H^1_0(\mathcal{D})$ to \eqref{semi-disc-cnd-problem-truncation}.
\end{proof}

 \subsection{Uniform $L^{\infty}$-bounds for the truncated approximations}

Goal of this section is to show that if $M =\| u \|_{L^{\infty}(\mathcal{D} \times (0,T))} + C\| u^0 \|_{L^{2}(\mathcal{D})}$, then there exists $\hat{\deltat{}}>0$ such that $\sup_{0\le n \le N} \| \hatuM{n} \|_{L^{\infty}(\mathcal{D})} \le M$ whenever $\sup_{0\le n \le N}\deltat{n}<\hat{\deltat{}}$.

The key to deriving such an $L^{\infty}$-bound is to first establish a uniform bound for the
error between the Laplacian of the exact solution $u$ and the Laplacian of a truncated approximation $\hatuM{}$, i.e. for  $\| \triangle \hatuM{n} -  \triangle u^n\|_{L^2(\mathcal{D})}$. We start with deriving corresponding error identities. For that purpose we define the continuous function $\GM$ for $t_1,t_2 \in \R_{\ge 0}$ by
\begin{align}\label{def-GM}
\GM(t_1,t_2):=
\begin{cases}
\frac{\Gamma_M(t_1) - \Gamma_M(t_2)}{t_1 - t_2} & \mbox{for  } t_1 \neq t_2 \\
\hspace{12pt}\gamma_M(t_1) & \mbox{for  } t_1 = t_2.
\end{cases}
\end{align}

\begin{lemma}[Error identities]
\label{lemma-error-ids}
Let $\hatuM{n} \in H^1_0(\mathcal{D})$ denote a solution to \eqref{semi-disc-cnd-problem-truncation} and
let $\hateM{0} := \hatuM{0} -  u^0=0$. For $n \in \N$, $n\ge1$ we define the error $\hateM{n} := \hatuM{n} -  u^n$. With $\ham:=\triangle - V$ it holds
\begin{eqnarray}
\nonumber\label{error-identity-L2}\lefteqn{ \| \hateM{n} \|^2_{L^2({\mathcal{D}})} = \| \hateM{n-1} \|^2_{L^2({\mathcal{D}})} + \frac{1}{2}\Im \left( \int_{I_n} \int_{\mathcal{D}} 
\ham \hspace{-2pt}
\left( 2u -  u^{n} - u^{n-1}\right) \left( \hateM{n}+ \hateM{n-1} \right) \right) } \\
\nonumber&\enspace& \enspace + 
\frac{1}{2} \Im \left( \int_{I_n} \int_{\mathcal{D}} 
\left(  \GM(|u^{n}|^2,|u^{n-1}|^2)  (u^{n} + u^{n-1}) - 2 \GM(|u|^2,|u|^2) u \right) \left( \hateM{n}+ \hateM{n-1} \right)
 \right)
\\
&\enspace& \enspace + 
\frac{1}{4} \deltat{n} \Im 
 \int_{\mathcal{D}} \left(
 \GM(|\hatuM{n}|^2,|\hatuM{n-1}|^2)  (\hatuM{n} + \hatuM{n-1}) \right.\\
 \nonumber&\enspace& \hspace{80pt} \left.  -
 \GM(|u^{n}|^2,|u^{n-1}|^2) (u^{n} + u^{n-1})
\left( \hateM{n}+ \hateM{n-1} \right) \right)
\end{eqnarray}
and
\begin{eqnarray*}
\nonumber\label{error-identity-L2-Laplacian}\lefteqn{\| \ham \hateM{n} \|_{L^2(\mathcal{D})}^2
= \| \ham \hateM{n-1} \|_{L^2(\mathcal{D})}^2
+ \frac{1}{\deltat{n}} \Re \int_{I_n} \int_{\mathcal{D}}  
\ham
\hspace{-2pt}\left(  u^{n} + u^{n-1} - 2u \right)  \left( \ham( \hateM{n} - \hateM{n-1}) \right)
 }\\
\nonumber&\enspace& \quad 
+ \frac{1}{\deltat{n}} \Re \int_{I_n} \int_{\mathcal{D}} \left( 
  \GM(|u^{n}|^2,|u^{n-1}|^2) (u^{n} + u^{n-1}) - 2 \GM(|u|^2,|u|^2) u \right) \left( \ham( \hateM{n} - \hateM{n-1} ) \right)  \\
&\enspace& \quad 
+ \Re \int_{\mathcal{D}} \left( \GM(|\hatuM{n}|^2,|\hatuM{n-1}|^2)  (\hatuM{n} + \hatuM{n-1}) \right. \\
\nonumber&\enspace& \hspace{80pt} \left.
 - \GM(|u^{n}|^2,|u^{n-1}|^2) (u^{n} + u^{n-1}) \right) \left( \ham ( \hateM{n} - \hateM{n-1} )
\right).
\end{eqnarray*}
\end{lemma}

\begin{proof}
For simplicity let $(v,w)_{H^1_V}:=\Ltwo{\nabla v }{\nabla w} +\Ltwo{V \hspace{2pt} v }{ w}$ for $v,w \in H^1_0(\mathcal{D})$.
From \eqref{semi-disc-cnd-problem-truncation} we have
\begin{eqnarray*}
\lefteqn{\Ltwo{ \hatuM{n} - \hatuM{n-1}}{v} +
\frac{\deltat{n}}{2} \hspace{2pt}\ci \hspace{2pt}
(\hatuM{n} + \hatuM{n-1}, v )_{H^1_V}
}\\
 &=& - \deltat{n} \hspace{2pt}\ci \hspace{2pt} \frac{1}{2} \Ltwo{ \GM(|\hatuM{n}|^2,|\hatuM{n-1}|^2) (\hatuM{n} + \hatuM{n-1}) }{v}.\hspace{140pt}
\end{eqnarray*}
Subtracting the term
\begin{eqnarray*}
\Ltwo{ u^{n} - u^{n-1} }{v } +
\frac{\deltat{n}}{2} \hspace{2pt} \ci \hspace{2pt} 
( u^{n} + u^{n-1} , v )_{H^1_V}
\end{eqnarray*}
on both sides gives us
\begin{eqnarray}
\nonumber\lefteqn{\Ltwo{ \hateM{n} - \hateM{n-1}}{v } +
\frac{\deltat{n}}{2} \hspace{2pt}\ci \hspace{2pt}
( \hateM{n} +  \hateM{n-1}, v )_{H^1_V}}\\
\nonumber &=& \Ltwo{ u^{n-1} }{v } - \Ltwo{ u^{n} }{v }
-\frac{\deltat{n}}{2} \hspace{2pt} \ci \hspace{2pt}
(u^{n} + u^{n-1} , v )_{H^1_V}
 \\
\nonumber&\enspace& \quad - \deltat{n} \hspace{2pt}\ci \hspace{2pt} \frac{1}{2} \Ltwo{ \GM(|\hatuM{n}|^2,|\hatuM{n-1}|^2) (\hatuM{n} + \hatuM{n-1}) }{v} \\
\nonumber &\overset{\eqref{exact-solution-with-discrete-test-func}}{=}& 
\ci \int_{I_n} \left( ( u(\cdot,t) , v )_{H^1_V} -\frac{1}{2} \hspace{2pt} ( u^{n} + u^{n-1} , v )_{H^1_V} \right) \hspace{2pt} dt
\\
\nonumber&\enspace& \quad
+ \ci \int_{I_n} \left( 
\Ltwo{  \GM(|u(\cdot,t)|^2,|u(\cdot,t)|^2) u(\cdot,t) }{ v }  \right.\\
\nonumber &\enspace& \left. \hspace{100pt}
-   \Ltwo{ \GM(|\hatuM{n}|^2,|\hatuM{n-1}|^2)  \frac{\hatuM{n} + \hatuM{n-1}}{2} }{v} 
\right) \hspace{2pt} dt \\
\nonumber &=& 
\ci \int_{I_n} \left( ( u(\cdot,t) , v )_{H^1_V} -\frac{1}{2} \hspace{2pt} ( u^{n} + u^{n-1} , v )_{H^1_V} \right) \hspace{2pt} dt
\\
\nonumber&\enspace& \quad 
+ \ci \int_{I_n} \int_{\mathcal{D}} \left( \GM(|u|^2,|u|^2) u - \GM(|u^{n}|^2,|u^{n-1}|^2)  \frac{u^{n} + u^{n-1}}{2} \right) v \\
\nonumber&\enspace& \quad 
+\ci \deltat{n} \int_{\mathcal{D}} \left( \GM(|u^{n}|^2,|u^{n-1}|^2)  \frac{u^{n} + u^{n-1}}{2}  -
 \GM(|\hatuM{n}|^2,|\hatuM{n-1}|^2)  \frac{\hatuM{n} + \hatuM{n-1}}{2} \right) v.
\end{eqnarray}
Testing with $v=\hateM{n}+ \hateM{n-1}$ and only using the real part of the equation proves the $L^2$-norm identity \eqref{error-identity-L2}. Testing with 
$v=\ham(\hateM{n} - \hateM{n-1})$
(note that $v\in L^2(\mathcal{D})$ is admissible here)
and taking the imaginary part proves the error identity for 
$\ham$ 
i.e. equation \eqref{error-identity-L2-Laplacian}.
\end{proof}

\begin{lemma}\label{lemma-bound-laplacian-hateM}
Let $\hatuM{n} \in H^1_0(\mathcal{D})$ denote a solution to \eqref{semi-disc-cnd-problem-truncation} and recall the definition of $\gamma^{M,k}$ from Lemma \ref{truncation-lemma}. For $C_{M,\gamma} := \left(\sum_{k=0}^2 M^{2k} \gamma^{M,k}\right)^{-1}$ there exists a generic constant $C_g$ (independent of the problem and the discretization) such that if $\deltat{} \le C_g C_{M,\gamma}$, if $\partial_t u \in L^4(\mathcal{D} \times (0,T))$, $\partial_{tt} u \in L^2(\mathcal{D} \times (0,T))$ and $\partial_{t}^{s} u \in L^2(0,T;H^2(\mathcal{D}))$
for $s=1$ or $s=2$,
 then it holds
\begin{align*}
\deltat{} \|  \ham \hateM{n} \|_{L^2(\mathcal{D})} + \| \hateM{n} \|_{L^2(\mathcal{D})} 
\le C(u,\gamma,V,\mathcal{D},M) \deltat{}^s.
\end{align*}
for some constant depending on $u$, $\gamma$, $V$, $ \mathcal{D}$ and $M$. Furthermore, for $s=1$ it still holds
\begin{align*}
\lim_{\deltat{}\rightarrow 0} \| \ham \hateM{n} \|_{L^2(\mathcal{D})} = 0.
\end{align*}
\end{lemma}

\begin{proof}
Let either $s=1$ or $s=2$.
We estimate 
\begin{eqnarray*}
\nonumber\lefteqn{\|  \ham \hateM{n} \|_{L^2(\mathcal{D})}^2
= \|  \ham \hateM{n-1} \|_{L^2(\mathcal{D})}^2
+ \underset{=:\mbox{I}}{\underbrace{\frac{1}{\deltat{n}} \Re \int_{I_n} \int_{\mathcal{D}} 
\ham
\hspace{-2pt}\left(  u^{n} + u^{n-1} - 2u \right)  \ham \hspace{-2pt}\left( \hateM{n} - \hateM{n-1} \right)} }
 }\\
\nonumber&\enspace& \quad 
+ \underset{=:\mbox{II}}{\underbrace{\frac{1}{\deltat{n}} \Re \int_{I_n} \int_{\mathcal{D}} \left( 
  \GM(|u^{n}|^2,|u^{n-1}|^2) (u^{n} + u^{n-1}) - 2 \GM(|u|^2,|u|^2) u \right) \ham \hspace{-2pt}\left( \hateM{n} - \hateM{n-1} \right)}}  \\
\nonumber&\enspace& \quad 
+ \Re \int_{\mathcal{D}} \left( \GM(|\hatuM{n}|^2,|\hatuM{n-1}|^2)  (\hatuM{n} + \hatuM{n-1}) \right. \\
\nonumber&\enspace& \hspace{20pt} \underset{=:\mbox{III}}{\underbrace{\hspace{60pt} \left.
 - \GM(|u^{n}|^2,|u^{n-1}|^2) (u^{n} + u^{n-1}) \right) \ham \hspace{-2pt} \left( \hateM{n} - \hateM{n-1}
\right)}}.
\end{eqnarray*}
For the first term we obtain with the trapezoidal rule for fixed $x$ that
\begin{align*}
\left| \int_{I_n} \ham \hspace{-2pt}\left(  u^{n}(x) + u^{n-1}(x) - 2u(x,\cdot) \right) \right|
\le  C \deltat{n}^{(2s+1)/2} \|  \partial_{t}^{s} \ham u(x,\cdot) \|_{L^2(I_n)}.
\end{align*}
Hence for $C=C(V)$
\begin{align*}
\mbox{I} 
&\le C \deltat{n}^{(2s-1)/2}  \| \partial_{t}^{s} u \|_{L^2(I_n, H^2(\mathcal{D}))} \left( \| \ham \hateM{n} \|_{L^2(\mathcal{D})} + \| \ham \hateM{n-1} \|_{L^2(\mathcal{D})} \right) \\
&\le C \deltat{n}^{2(s-1)} \|  \partial_{t}^{s} u \|_{L^2(I_n, H^2(\mathcal{D}))}^2 + \frac{1}{3} \deltat{n} \left( \| \ham \hateM{n} \|^2_{L^2(\mathcal{D})} + \| \ham \hateM{n-1} \|^2_{L^2(\mathcal{D})} \right).
\end{align*}
In order to estimate the second term we split the error into several contributions.
\begin{eqnarray*}
\lefteqn{ \frac{1}{2} \left| \int_{I_n} \left( \GM(|u^{n}|^2,|u^{n-1}|^2) (u^{n} + u^{n-1}) - 2 \GM(|u|^2,|u|^2) u \right) \right| }\\
&=&  \frac{1}{2} \left| \int_{I_n} \left( \frac{\Gamma_M(|u^{n}|^2) - \Gamma_M(|u^{n-1}|^2)}{|u^{n}|^2 - |u^{n-1}|^2} (u^{n} + u^{n-1}) - 2 \gamma_M(|u|^2) u \right) \right|\\
&\le&
\underset{=:\mbox{II}_1}{\underbrace{\left| \int_{I_n} \left( \frac{\Gamma_M(|u^{n}|^2) - \Gamma_M(|u^{n-1}|^2)}{|u^{n}|^2 - |u^{n-1}|^2} \frac{u^{n} + u^{n-1}}{2} - \gamma_M\left( \left|\frac{u^{n}+u^{n-1}}{2}\right|^2\right) \frac{u^{n}+u^{n-1}}{2} \right) \right|}}\\
&\enspace& \quad +
\underset{=:\mbox{II}_2}{\underbrace{\left| \int_{I_n} \left( \gamma_M\left( \left|\frac{u^{n}+u^{n-1}}{2}\right|^2\right) \frac{u^{n}+u^{n-1}}{2} - \frac{ \gamma_M(|u^{n}|^2) u^{n} + \gamma_M(|u^{n-1}|^2) u^{n-1} }{ 2 } \right) \right|}}\\
\\
&\enspace& \quad +  
\underset{=:\mbox{II}_3}{\underbrace{\left| 
\int_{I_n} \left( 
\frac{ \gamma_M(|u^{n}|^2) u^{n} + \gamma_M(|u^{n-1}|^2) u^{n-1} }{ 2 }
- \gamma_M(|u|^2) u
 \right)
 \right|}}.
\end{eqnarray*}
For $\mbox{II}_1$ we can apply Lemma \ref{lemma01ap} to obtain
\begin{align*}
\mbox{II}_1 \le \deltat{n} |u^n-u^{n-1}|^2 \left( \frac{1}{4} M \gamma^{M,1} + \frac{4}{3} M^3 \gamma^{M,2} \right).
\end{align*}
For $\mbox{II}_2$ we 
denote $f(z):= \gamma_M(|z|^2) z$ and we let $\zeta_n: [0,1] \rightarrow [u^{n-1},u^{n}]$ denote the complex valued (linear) curve given by
$\zeta_n(s):=(1-s) u^{n-1} + s u^{n}$ for $s \in [0,1]$. We have $\zeta_n^{\prime}(z)=u^{n}-u^{n-1}$. With that, we get with the trapezoidal-rule and the midpoint rule that
\begin{eqnarray}
\nonumber\lefteqn{\mbox{II}_2 = \deltat{n} \left| \frac{f(u^{n}) +f(u^{n-1})}{2}  - f\left( \frac{u^{n}+u^{n-1}}{2} \right)\right|}\\
\nonumber&\le& \deltat{n} \left|  \frac{(f\circ\zeta_n)(0) +(f\circ\zeta_n)(1)}{2}  
- \int_{0}^{1} (f \circ \zeta_n)(s) \hspace{2pt} ds 
\right| + 
\deltat{n} \left|  \int_{0}^{1} (f \circ \zeta_n)(s) \hspace{2pt} ds  - (f \circ \zeta_n)\left( \frac{1}{2} \right) \right|\\
\nonumber &\le& \frac{\deltat{n}}{12} \| (f \circ \zeta_n)^{\prime \prime} \|_{L^{\infty}(0,1)} + \frac{\deltat{n}}{24} \| (f \circ \zeta_n)^{\prime \prime} \|_{L^{\infty}(0,1)} =  \frac{\deltat{n}}{8} \| (f \circ \zeta_n)^{\prime \prime} \|_{L^{\infty}(0,1)}\\
\nonumber &\le& C \deltat{n} \left( \gamma^{M,1} M + \gamma^{M,2} M^3 \right) |u^{n} - u^{n-1}|^2. 
\end{eqnarray}
For term $\mbox{II}_3$ we can directly apply the trapezoidal rule again to obtain
\begin{align*}
\mbox{II}_3
\le  C \deltat{n}^{5/2} \| \partial_{tt} \left( \gamma(|u(x,\cdot)|^2) u(x,\cdot) \right) \|_{L^2(I_n)}.
\end{align*}
Combining the estimates for $\mbox{II}_1$, $\mbox{II}_2$ and $\mbox{II}_3$ and applying the Young-inequality yields
\begin{eqnarray*}
\lefteqn{\mbox{II}
\le C 
\frac{1}{\deltat{n}} \left( \gamma^{M,1} M + \gamma^{M,2} M^3 \right)^2 \| u^{n} - u^{n-1} \|_{L^4(\mathcal{D})}^4 
+ C \| \deltat{} \hspace{2pt}\partial_{tt} \hspace{-2pt}\left( \gamma(|u|^2) u \right) \hspace{-2pt} \|_{L^2(\mathcal{D} \times I_n)}^2}\\
&\enspace& \qquad
+ \frac{1}{3} \deltat{n}
\|  \ham \hateM{n}  \|_{L^2(\mathcal{D})}^2 + \frac{1}{3} \deltat{n} \| \ham \hateM{n-1} \|_{L^2(\mathcal{D})}^2 \hspace{250pt}\\
&\le& C \left( \left( \gamma^{M,1} M + \gamma^{M,2} M^3 \right)^2 \| \sqrt{\deltat{}} \hspace{2pt} \partial_t u \|_{L^4(\mathcal{D} \times I_n)}^4
+ \| \deltat{} \hspace{2pt}\partial_{tt} \hspace{-2pt}\left( \gamma(|u|^2) u \right) \hspace{-2pt} \|_{L^2(\mathcal{D} \times I_n)}^2 \right)\\
&\enspace& \qquad
+ \frac{1}{3} \deltat{n}
\|  \ham \hateM{n}  \|_{L^2(\mathcal{D})}^2 + \frac{1}{3} \deltat{n} \| \ham \hateM{n-1} \|_{L^2(\mathcal{D})}^2 \hspace{250pt},
\end{eqnarray*}
where we used $\| u^{n} - u^{n-1} \|_{L^4(\mathcal{D})}^4 = \int_{\mathcal{D}} \left( \int_{I_n} \partial_t u \right)^4 \le \deltat{n} \| \sqrt{\deltat{}} \hspace{2pt} \partial_t u \|_{L^4(\mathcal{D} \times I_n)}^4$.
It remains to estimate term $\mbox{III}$ for which we can use $|\GM(\cdot,\cdot)|\le \gamma^{M,0}$ to see
\begin{eqnarray*}
\lefteqn{\mbox{III} \le \gamma^{M,0}  \int_{\mathcal{D}}  \left|  \hateM{n} + \hateM{n-1}
\right| \hspace{2pt} \left| \ham \hateM{n} - \ham \hateM{n-1}
\right|} \\
&\enspace& +  \int_{\mathcal{D}} \left|
\GM(|\hatuM{n}|^2,|\hatuM{n-1}|^2) - \GM(|u^{n}|^2,|u^{n-1}|^2) ( u^{n} + u^{n-1}) \right| \left| \ham \hateM{n} - \ham \hateM{n-1}
\right|\\
&\overset{\eqref{f_M_cond_3}}{\le}& \gamma^{M,0}  \| \hateM{n} + \hateM{n-1} \|_{L^2(\mathcal{D})} \hspace{2pt} \| \ham \hateM{n} - \ham \hateM{n-1} \|_{L^2(\mathcal{D})} \\
&\enspace& + C  
 \left(\sum_{k=1}^2 M^{2k-1} \gamma^{M,k}\right)
\int_{\mathcal{D}} \left| u^{n} - u^{n-1}  \right|^2 \left| \ham \hateM{n} - \ham \hateM{n-1}
\right|\\
&\enspace& + C \left(\sum_{k=0}^2 M^{2k} \gamma^{M,k}\right) \int_{\mathcal{D}} \left|  |\hateM{n}| + |\hateM{n-1}| \right| \left| \ham \hateM{n} - \ham \hateM{n-1}
\right|\\
&\le& C \deltat{n}^{-1} \left(\gamma^{M,0} + \sum_{k=0}^2 M^{2k} \gamma^{M,k}\right)^2 \left( \| \hateM{n} \|^2_{L^2(\mathcal{D})} + \| \hateM{n-1} \|^2_{L^2(\mathcal{D})} \right)  \\
&\enspace& + C  
 \left(\sum_{k=1}^2 M^{2k-1} \gamma^{M,k}\right)^2
\| \sqrt{\deltat{}} \hspace{2pt} \partial_t u \|_{L^4(\mathcal{D} \times I_n)}^4 + \frac{1}{3} \deltat{n} \hspace{2pt} \left( \| \ham \hateM{n} \|_{L^2(\mathcal{D})}^2 + \| \ham \hateM{n-1} \|_{L^2(\mathcal{D})}^2 \right).
\end{eqnarray*}
Combining the estimates for I, II and III yields
\begin{eqnarray}
\label{preliminary-estimate-laplace-hateM}\nonumber\lefteqn{(1 - \deltat{n}) \|  \ham \hateM{n} \|_{L^2(\mathcal{D})}^2
\le (1 + \deltat{n}) \|  \ham \hateM{n-1} \|_{L^2(\mathcal{D})}^2 + C \| \deltat{}^{s-1} \partial_{t}^{s} u \|_{L^2(I_n, H^2(\mathcal{D}))}^2 }\\
&\enspace&+ C \deltat{n}^{-1} \left(\sum_{k=0}^2 M^{2k} \gamma^{M,k}\right)^2 \left( \| \hateM{n} \|^2_{L^2(\mathcal{D})} + \| \hateM{n-1} \|^2_{L^2(\mathcal{D})} \right) \\
\nonumber&\enspace& + C  
 \left(\sum_{k=1}^2 M^{2k-1} \gamma^{M,k}\right)^2
\| \sqrt{\deltat{}} \hspace{2pt} \partial_t u \|_{L^4(\mathcal{D} \times I_n)}^4 + C \| \deltat{} \hspace{2pt}\partial_{tt} \hspace{-2pt}\left( \gamma(|u|^2) u \right) \hspace{-2pt} \|_{L^2(\mathcal{D} \times I_n)}^2.
\end{eqnarray}
Next, we need to show that $\deltat{n}^{-1} \| \hateM{n} \|^2_{L^2(\mathcal{D})}$ is an $\mathcal{O}(\deltat{})$-term. For that reason, we start from the identity \eqref{error-identity-L2} this time, where we observe that we have just the desired $\mathcal{O}(\deltat{n})$ more in our terms. Proceeding analogously as before yields
\begin{eqnarray}
\nonumber\lefteqn{(1 - C \deltat{n} C_{M,\gamma}) \| \hateM{n} \|_{L^2(\mathcal{D})}^2
\le (1 +C \deltat{n} C_{M,\gamma} ) \| \hateM{n-1} \|_{L^2(\mathcal{D})}^2 + C 
\| \deltat{}^{s} \partial_{t}^{s} u \|_{L^2(I_n, H^2(\mathcal{D}))}^2
}\\
\label{L2-estimate-preliminary-with-bound}&\enspace& + C  
 \left(\sum_{k=1}^2 M^{2k-1} \gamma^{M,k}\right)^2
\| \deltat{} \hspace{2pt} \partial_t u \|_{L^4(\mathcal{D} \times I_n)}^4 + C \| \deltat{}^2 \hspace{2pt}\partial_{tt} \hspace{-2pt}\left( \gamma(|u|^2) u \right) \hspace{-2pt} \|_{L^2(\mathcal{D} \times I_n)}^2.
\end{eqnarray}
Assuming that $\deltat{n}$ is small enough so that $C \deltat{n} C_{M,\gamma} \le \frac{1}{2}$ we can divide by $(1 - C \deltat{n} C_{M,\gamma})$ on both sides of \eqref{L2-estimate-preliminary-with-bound}. Applying the arising inequality iteratively and exploiting that $\hateM{0}=0$ yields
\begin{multline*}
\| \hateM{n} \|_{L^2(\mathcal{D})}^2 \le C e^{C t_{n-1} C_{M,\gamma}} \Biggl(\| \deltat{}^{s} \partial_{t}^{s} u \|_{L^2(0,t_n; H^2(\mathcal{D}))}^2 + \biggl(\sum_{k=1,2} M^{2k-1} \gamma^{M,k}\biggr)^2
\| \deltat{} \hspace{2pt} \partial_t u \|_{L^4(\mathcal{D} \times (0,t_n) )}^4\\
+
\| \deltat{}^2 \hspace{2pt}\partial_{tt} \hspace{-2pt}\left( \gamma(|u|^2) u \right) \hspace{-2pt} \|_{L^2(\mathcal{D} \times(0,t_n) )}^2
  \Biggr).
\end{multline*}
Now we can plug this estimate into \eqref{preliminary-estimate-laplace-hateM} to get
\begin{eqnarray*}
\lefteqn{\|  \ham \hateM{n} \|_{L^2(\mathcal{D})}^2
\le \frac{(1 + \deltat{n})}{(1 - \deltat{n}) } \|  \ham \hateM{n-1} \|_{L^2(\mathcal{D})}^2 + C 
\| \deltat{}^{s-1 \partial_{t}^{s}} u \|_{L^2(I_n, H^2(\mathcal{D}))}^2
}\\
&\enspace&+ C \deltat{n} \left(\sum_{k=0}^2 M^{2k} \gamma^{M,k}\right)^2 e^{C t_{n-1} C_{M,\gamma}} 
\| \deltat{}^s \partial_{t}^{{s-1}} u \|_{L^2(0,t_n; H^2(\mathcal{D}))}^2
\\
&\enspace&+ C \deltat{n} \left(\sum_{k=0}^2 M^{2k} \gamma^{M,k}\right)^2
e^{C t_{n-1} C_{M,\gamma}} \left(\sum_{k=1}^2 M^{2k-1} \gamma^{M,k}\right)^2
\| \sqrt{\deltat{}} \hspace{2pt} \partial_t u \|_{L^4(\mathcal{D} \times (0,t_n) )}^4 \\
&\enspace&+ C \deltat{n} \left(\sum_{k=0}^2 M^{2k} \gamma^{M,k}\right)^2  e^{C t_{n-1} C_{M,\gamma}}
\| \deltat{} \partial_{tt} \hspace{-2pt}\left( \gamma(|u|^2) u \right) \hspace{-2pt} \|_{L^2(\mathcal{D} \times(0,t_n) )}^2 \\
&\enspace& + C  
 \left(\sum_{k=1}^2 M^{2k-1} \gamma^{M,k}\right)^2
\| \sqrt{\deltat{}} \hspace{2pt} \partial_t u \|_{L^4(\mathcal{D} \times I_n)}^4 + C \| \deltat{} \hspace{2pt}\partial_{tt} \hspace{-2pt}\left( \gamma(|u|^2) u \right) \hspace{-2pt} \|_{L^2(\mathcal{D} \times I_n)}^2.
\end{eqnarray*}
Observe that this step exploits the quasi-uniformity of the time-discretization. Applying this inequality iteratively gives us
\begin{multline}\label{gtum}
\|  \ham \hateM{n} \|_{L^2(\mathcal{D})}
\lesssim\\ e^{t_n} 
\| \deltat{}^{s-1} \partial_{t}^{{s}} u \|_{L^2(0,t_n; H^2(\mathcal{D}))}
+ \left(\sum_{k=0}^2 M^{2k} \gamma^{M,k}\right) e^{C t_{n} (1+C_{M,\gamma})} 
\| \deltat{}^{s-1} \partial_{t}^{{s}} u \|_{L^2(0,t_n; H^2(\mathcal{D}))}
\\
+ \left(1 + \sum_{k=0}^2 M^{2k} \gamma^{M,k}\right)
e^{C t_{n} (1+C_{M,\gamma})} \left(\sum_{k=1}^2 M^{2k-1} \gamma^{M,k}\right)
\| \sqrt{\deltat{}} \hspace{2pt} \partial_t u \|_{L^4(\mathcal{D} \times (0,t_n) )}^2
\\
+ \left( 1+ \sum_{k=0}^2 M^{2k} \gamma^{M,k}\right) e^{C t_{n} (1+C_{M,\gamma})}
\| \deltat{} \partial_{tt} \hspace{-2pt}\left( \gamma(|u|^2) u \right) \hspace{-2pt} \|_{L^2(\mathcal{D} \times(0,t_n) )}.
\end{multline}
This proves the estimates in the lemma. 
It remains to verify $\lim_{\deltat{}\rightarrow 0} \| \ham \hateM{n} \|_{L^2(\mathcal{D})} = 0$ for $s=1$. 
Checking the previous estimates carefully, we see that the problematic terms $\| \partial_{t} u \|_{L^2(0,t_n; H^2(\mathcal{D}))}$ (which prevent the right hand side of \eqref{gtum} to converge to zero) can be replaced by $\deltat{}^{-1} |T_{\deltat{}}(\ham u)|$, where $T_{\deltat{}}(\ham u)$ is the error introduced by the trapezoidal rule applied to the individual integrals $\int_{t_{n-1}}^{t_n} \ham u$ (and summed up). It can be shown that the error $\deltat{}^{-1} |T_{\deltat{}}(\ham u)|$ converges to zero without additional regularity assumption, cf. \cite[Theorem 1.13]{CrN02}. This finishes the proof of the lemma.
\end{proof}

\begin{remark}
It holds
\begin{align*}
| \partial_{tt} \left( \gamma(|u|^2) u \right) |  &\le
  \gamma(|u|^2) |\partial_{tt} u|
  + 2 \gamma^{\prime}(|u|^2) |u| \left( |\partial_{tt} u| \hspace{2pt} |u| + 3 |\partial_t u|^2 \right)
  + 4 \gamma^{\prime\prime}(|u|^2) \hspace{2pt} |u|^3 |\partial_t u |^2,
\end{align*}
Hence, $\| \partial_{tt} \left( \gamma(|u|^2) u \right) \|_{L^2(\mathcal{D} \times (0,T))}$ can be bounded if $\partial_t u \in L^4(\mathcal{D} \times (0,T))$ and
$\partial_{tt} u \in L^2(\mathcal{D} \times (0,T))$. Note that the latter one implies $u\in L^{\infty}(0,T;L^2(\mathcal{D}))$.
\end{remark}

\begin{lemma}
\label{lemma-L-infty-bounds-for-semidiscrete-trunc-approx}
Consider the setting of Lemma \ref{lemma-bound-laplacian-hateM}. There exist constants $C_{\mathcal{D}}$ that only depends on $\mathcal{D}$ and $\hat{\deltat{}}>0$ such that any solution $\hatuM{n}$ to \eqref{semi-disc-cnd-problem-truncation} fulfills the bound
\begin{align*}
\sup_{0\le n \le N} \| \hatuM{n} \|_{L^{\infty}(\mathcal{D})} \le 2 \| u \|_{L^{\infty}(\mathcal{D} \times (0,T))} 
=:M
\end{align*}
for all partitions with $\sup_{0\le n \le N}\deltat{n}<\hat{\deltat{}}$.
\end{lemma}
\begin{proof}
Since $\mathcal{D}$ is convex, $d\le 3$ and since the potential $V \in L^{\infty}(\mathcal{D})$ is non-negative, we have from elliptic regularity theory that for any $w \in H^1_0(\mathcal{D}) \cap H^2(\mathcal{D})$ with $- \triangle w + Vw = f(w) \in L^2(\mathcal{D})$ the solution $w$ is continuous on $\overline{\mathcal{D}}$ and it holds (cf. \cite{GiT01})
\begin{align}
\label{bound-lmma-stp-2}
\| w \|_{L^{\infty}(\mathcal{D})} \le C(\mathcal{D},V) \| f(w) \|_{L^2(\mathcal{D})} = C(\mathcal{D},V) \| \ham w \|_{L^2(\mathcal{D})},
\end{align}
where $C(\mathcal{D},V)$ only depends on $\mathcal{D}$, $V$ and $d$.
Since $\ham \hateM{n} \in L^2(\mathcal{D})$ we can apply \eqref{bound-lmma-stp-2} together with Lemma \ref{lemma-bound-laplacian-hateM} to conclude that for every $\eps>0$ there exists a $\hat{\deltat{}}(\eps)$ such that for all $\deltat{}\le \hat{\deltat{}}(\eps)$ it holds
\begin{align}
\label{bound-lmma-stp-3}
\sup_{0\le n \le N} \| \hateM{n} \|_{L^{\infty}(\mathcal{D})} \le C(\mathcal{D},V) \sup_{0\le n \le N} \|  \ham \hateM{n} \|_{L^2(\mathcal{D})} \le \eps.
\end{align}
This implies
\begin{eqnarray*}
\sup_{0\le n \le N} \| \hatuM{n} \|_{L^{\infty}(\mathcal{D})} \le 
\sup_{0\le n \le N} \| u^n \|_{L^{\infty}(\mathcal{D})} +\sup_{0\le n \le N} \| \hateM{n} \|_{L^{\infty}(\mathcal{D})}
\le \sup_{0\le n \le N} \| u^n \|_{L^{\infty}(\mathcal{D})} + \eps. 
\end{eqnarray*}
The choice $\eps=\sup_{0\le n \le N} \| u^n \|_{L^{\infty}(\mathcal{D})}$ proves the lemma.
\end{proof}

\subsection{Existence of uniformly $L^{\infty}$-bounded solutions to the semi-discrete scheme and corresponding $L^2$-error estimates}

We are now prepared to proceed with the analysis of the original (non-truncated) semi-discrete scheme.
\begin{theorem}
\label{main-theorem-1}
Let $\partial_t u \in L^4(\mathcal{D} \times (0,T))$, $\partial_{tt} u \in L^2(\mathcal{D} \times (0,T))$ and $\partial_{t}^{s} u \in L^2(0,T;H^2(\mathcal{D}))$
for $s=1$ or $s=2$.
There is a real number $\hat{\deltat{}}>0$ such that for all partitions with $\sup_{0\le n \le N}\deltat{n}<\hat{\deltat{}}$ there exists a unique solution $\hatu^{n}\in H^1_0(\mathcal{D})$ to the semi-discrete Crank-Nicolson scheme \eqref{semi-disc-cnd-problem} with
\begin{align*}
\sup_{0\le n \le N} \| \hatu^{n} \|_{L^{\infty}(\mathcal{D})} \le M
\end{align*}
where $ M:= 2 \| u \|_{L^{\infty}(\mathcal{D} \times (0,T))}$. 
Moreover, the uniform $H^2$ bound 
\begin{align*}
\sup_{0\le n \le N} \| \hatu^{n} \|_{H^{2}(\mathcal{D})} \le M + \sup_{0\le n \le N} \| u^n \|_{H^{2}(\mathcal{D})} 
\end{align*}
and the error estimate
\begin{align*}
\sup_{0\le n \le N} \| \hatu^{n} - u^n \|_{L^{2}(\mathcal{D})}  + \deltat{} \sup_{0\le n \le N} \left( \| \hatu^{n} - u^n \|_{L^{\infty}(\mathcal{D})} +  \| \hatu^{n} - u^n \|_{H^{2}(\mathcal{D})} \right)
\le C(u,\gamma,V,\mathcal{D}) \deltat{}^{{s}}
\end{align*}
hold true. Any other family of semi-discrete solutions must necessarily blow up in the sense that $\sup_{0\le n \le N} \| \hatu^{n} \|_{L^{\infty}(\mathcal{D})}\rightarrow \infty$ as $\deltat{}\rightarrow 0$.
\end{theorem}
\begin{proof}
From Lemmas \ref{existence-of-semi-disc-trunc-solutions} and \ref{lemma-L-infty-bounds-for-semidiscrete-trunc-approx} we immediately have the existence of $\hatu^{n}=\hatuM{n}$ and the uniform $L^{\infty}$-bound. The uniform $H^2$-bound follows from $\sup_{0\le n \le N}\| \hatu^{n} \|_{H^{2}(\mathcal{D})} \le \sup_{0\le n \le N}\| \hatu^{n} - u^n \|_{H^{2}(\mathcal{D})} + \sup_{0\le n \le N}\| u^n \|_{H^{2}(\mathcal{D})}$, where $\sup_{0\le n \le N}\| \hatu^{n} - u^n \|_{H^{2}(\mathcal{D})} \rightarrow 0$ for $\deltat{}\rightarrow 0$ with Lemma \ref{lemma-bound-laplacian-hateM}. The error estimates also follow directly from Lemma \ref{lemma-bound-laplacian-hateM} via \eqref{bound-lmma-stp-2}. It remains to show the uniqueness of $\hatu^{n}$. Let therefore $\hatu^{(1),n},\hatu^{(2),n} \in H^1_0(\mathcal{D})$ denote two solutions to the scheme \eqref{semi-disc-cnd-problem} for $n\ge1$ with $\| \hatu^{(1),n} \|_{L^{\infty}(\mathcal{D})},\| \hatu^{(2),n} \|_{L^{\infty}(\mathcal{D})} \le M$ and the same starting value, i.e. with $\hatu^{(1),n-1}=\hatu^{(2),n-1}=\hatu^{n-1}$ and $\| \hatu^{n-1} \|_{L^{\infty}(\mathcal{D})} \le M$.
By exploiting \eqref{semi-disc-cnd-problem} for $\hatu^{(1),n}$ and $\hatu^{(2),n}$ and by testing with $v=\hatu^{(1),n}-\hatu^{(2),n}$ we obtain
\begin{eqnarray*}
\lefteqn{ \| \hatu^{(1),n} - \hatu^{(2),n} \|_{L^{2}(\mathcal{D})}^2 }\\
&=& \frac{\deltat{n}}{2} \hspace{2pt}
\Im\Ltwo{\nabla (\hatu^{(1),n}+\hatu^{n-1}) - \nabla (\hatu^{(2),n}+\hatu^{n-1}) }{\nabla \hatu^{(1),n} - \nabla \hatu^{(2),n}}\\
&\enspace&
+
\frac{\deltat{n}}{2} \hspace{2pt}
\Im \Ltwo{V \hspace{2pt} \left((\hatu^{(1),n}+\hatu^{n-1}) - (\hatu^{(2),n}+\hatu^{n-1}) \right) }{ \hatu^{(1),n} - \hatu^{(2),n} } \\
&\enspace& + \deltat{n} \hspace{2pt} \Im \Ltwobig{ 
\frac{\Gamma(|\hatu^{(1),n}|^2) - \Gamma(|\hatu^{n-1}|^2)}{|\hatu^{(1),n}|^2 - |\hatu^{n-1}|^2} (\hatu^{(1),n} - \hatu^{(2),n})
}{  \hatu^{(1),n} - \hatu^{(2),n} }
\\
&\enspace& + \deltat{n} \hspace{2pt} \Im \Ltwobig{ 
\left( \frac{\Gamma(|\hatu^{(1),n}|^2) - \Gamma(|\hatu^{n-1}|^2)}{|\hatu^{(1),n}|^2 - |\hatu^{n-1}|^2}
-
\frac{\Gamma(|\hatu^{(2),n}|^2) - \Gamma(|\hatu^{n-1}|^2)}{|\hatu^{(2),n}|^2 - |\hatu^{n-1}|^2} \right) \hatu^{(2),n-\frac{1}{2}}
}{  \hatu^{(1),n} - \hatu^{(2),n} } \\
&=& \deltat{n} \hspace{2pt} \Im \Ltwobig{ 
\left( \frac{\Gamma(|\hatu^{(1),n}|^2) - \Gamma(|\hatu^{n-1}|^2)}{|\hatu^{(1),n}|^2 - |\hatu^{n-1}|^2}
-
\frac{\Gamma(|\hatu^{(2),n}|^2) - \Gamma(|\hatu^{n-1}|^2)}{|\hatu^{(2),n}|^2 - |\hatu^{n-1}|^2} \right) \hatu^{(2),n-\frac{1}{2}}
}{  \hatu^{(1),n} - \hatu^{(2),n} } \\
&\le& \deltat{n} \hspace{2pt} C(M) \| \hatu^{(1),n} - \hatu^{(2),n} \|_{L^{2}(\mathcal{D})}^2,
\end{eqnarray*}
which is a contradiction for $\deltat{n}<C(M)^{-1}$. Hence, we have uniqueness under the condition $\sup_{0\le n \le N} \| \hatu^{n} \|_{L^{\infty}(\mathcal{D})} \le M$. Observe that if there exists a another solution $\hat{u}_{\deltat{}}^{n}$ with $M<\sup_{0\le n \le N} \| \hat{u}_{\deltat{}}^{n} \|_{L^{\infty}(\mathcal{D})} \le \tilde{M} < \infty$, then the arguments remain the same and we conclude again uniqueness, which however then contradicts $\sup_{0\le n \le N} \| \hatu^{n} \|_{L^{\infty}(\mathcal{D})} < \sup_{0\le n \le N} \| \hat{u}_{\deltat{}}^{n} \|_{L^{\infty}(\mathcal{D})}$. From this we see that the only other solutions to \eqref{semi-disc-cnd-problem} are the ones with a diverging $L^{\infty}$-norm and which hence cannot approximate the smooth exact solution $u$.
\end{proof}

\section{Error analysis for the fully-discrete method}\label{s:errorfull}

In this section we shall analyze the final fully discrete Crank-Nicolson discretization stated in Definition \ref{crank-nic-gpe}. As for the semi-discrete method we need to take a detour over an auxiliary problem. For that reason we consider the following scheme with truncated nonlinearity.
\begin{definition}[Fully-discrete Crank-Nicolson Method with truncation]
\label{truncated-fully-discrete-crank-nic-gpe}
Let $M:= 2 \| u \|_{L^{\infty}(\mathcal{D} \times (0,T))}$ 
and let $\deltat{}$ be small enough so that the results of Theorem \ref{main-theorem-1} are valid.  For $\Gamma_M$ given as in Lemma \ref{truncation-lemma}, for $\hatuhM{0}:=P_h(u^0)\in \Sh$ and for $n \ge 1$, we define the truncated fully-discrete Crank-Nicolson approximation $\hatuhM{n} \in \Sh$ as the solution to
\begin{eqnarray}
\label{fully-disc-cnd-problem-truncation}\nonumber\lefteqn{\Ltwo{ \hatuhM{n} }{v_h} +
\deltat{n} \hspace{2pt}\ci \hspace{2pt}
\Ltwo{\nabla \hatuhM{n-\frac{1}{2}}}{\nabla v_h}
+
\deltat{n} \hspace{2pt}\ci \hspace{2pt}
\Ltwo{V \hatuhM{n-\frac{1}{2}}}{ v_h}}
\\
&+& \deltat{n} \hspace{2pt}\ci \hspace{2pt} \Ltwo{ \frac{\Gamma_M(|\hatuhM{n}|^2) - \Gamma_M(|\hatuhM{n-1}|^2)}{|\hatuhM{n}|^2 - |\hatuhM{n-1}|^2} \hatuhM{n-\frac{1}{2}}}{v_h} = \Ltwo{ \hatuhM{n-1}}{v_h} 
\end{eqnarray}
for all $v_h \in \Sh$ and where $\hatuhM{n-\frac{1}{2}}:=(\hatuhM{n}+\hatuhM{n-1})/2$.
\end{definition}
Again, we have existence of solutions.
\begin{lemma}
\label{existence-of-fully-disc-trunc-solutions}
For every $n\ge1$ there exists at least one solution $\hatuhM{n} \in \Sh$ to the truncated problem \eqref{fully-disc-cnd-problem-truncation}.
\end{lemma}
The proof is covered by the proof of Lemma \ref{existence-of-semi-disc-trunc-solutions}. 
Before deriving a $L^2$-error estimate, we need one last auxiliary lemma.
\begin{lemma}
\label{lemma-for-est-fully-disc}
For all $z_0,v_1,v_2 \in \C$ with $|v_\ell|\neq|z_0|$ for $\ell=1,2$, it holds
\begin{align*}
\left| \frac{\Gamma_M(|v_1|^2) - \Gamma_M(|z_0|^2)}{|v_1|^2 - |z_0|^2}
- \frac{\Gamma_M(|v_2|^2) - \Gamma_M(|z_0|^2)}{|v_2|^2 - |z_0|^2}
\right|
\le  \max\{ 4 M \gamma^{M,1} ,  \Gamma_M(M^2)  \}  |v_1 - v_2 |.
\end{align*}
\end{lemma}
\begin{proof}
We assume without loss of generality that $M\ge 1$. For fixed $|z_0|^2$ and $t\ge 0$ we investigate the function
\begin{align*}
g_M(t) := 
\begin{cases}
\frac{\Gamma_M(t) - \Gamma_M(|z_0|^2)}{t - |z_0|^2} & \mbox{for  } t \neq |z_0|^2 \\
\hspace{12pt}\gamma_M(|z_0|^2) & \mbox{for  } t = |z_0|^2.
\end{cases}
\end{align*}
We easily observe that $g_M$ is Lipschitz continuous because we have for some $\theta\in [0,1]$
\begin{align*}
\left| g^{\prime}_M(t) \right| &= \left| \frac{\Gamma_M^{\prime}(t) - \Gamma_M^{\prime}(|z_0|^2)}{t - |z_0|^2}
+ \left( t - |z_0|^2 \right)^{-1} \left( \Gamma_M^{\prime}(|z_0|^2) - \frac{\Gamma_M(t) - \Gamma_M(|z_0|^2)}{t - |z_0|^2} \right) \right| \\
&\le \gamma^{M,1} + \gamma^{M,1} \left|\frac{ |z_0|^2 - (\theta |z_0|^2 + (1-\theta) t) }{ t - |z_0|^2}\right|
= \gamma^{M,1} + (1-\theta) \gamma^{M,1} \le 2 \gamma^{M,1}.
\end{align*}
Now we investigate $|g_M(|v_1|^2) - g_M(|v_2|^2)|$ where we distinguish three cases.\\
Case 1: $| v_1 | \le 2M$ and $|v_2| \le 2M$. We obtain with the Lipschitz continuity of $g_M$
\begin{align*}
|g_M(|v_1|^2) - g_M(|v_2|^2)| \le  2 \gamma^{M,1} \left| |v_1|^2 - |v_2|^2 \right| \le 4 M \gamma^{M,1} \left| |v_1| - |v_2| \right|  \le 4 M \gamma^{M,1} \left| v_1 - v_2 \right| .
\end{align*}
Case 2: $| v_1 | \ge 2M$ and $|v_2| \ge 2M$ (and $|z_0|\le M$, otherwise everything is trivial). Without loss of generality let $| v_1 | \le |v_2|$ We obtain 
\begin{align*}
|g_M(|v_1|^2) - g_M(|v_2|^2)| 
&= 
\left| 
\frac{\Gamma_M(M^2) - \Gamma_M(|z_0|^2)}{|v_1|^2 - |z_0|^2}
- \frac{\Gamma_M(M^2) - \Gamma_M(|z_0|^2)}{|v_2|^2 - |z_0|^2}
\right| \\
&= \left( \Gamma_M(M^2) - \Gamma_M(|z_0|^2) \right)
\left| 
\frac{|v_2|^2 - |v_1|^2 }{(|v_1|^2 - |z_0|^2)(|v_2|^2 - |z_0|^2)}
\right| \\
&= \left( \Gamma_M(M^2) - \Gamma_M(|z_0|^2) \right)
\left| 
\frac{|v_2| + |v_1| }{(|v_1|^2 - |z_0|^2)(|v_2|^2 - |z_0|^2)}
\right| (|v_2| - |v_1|) \\
&\le \frac{\Gamma_M(M^2)}{M^2}
\left| 
\frac{|v_2| + |v_1| }{(|v_1| + |z_0|)(|v_2| + |z_0|)}
\right| (|v_2| - |v_1|) \le \Gamma_M(M^2) |v_1 - v_2|,
\end{align*}
where we used that $|v_2| - |v_1| \le |v_1-v_2|$; $M\ge1$ and that $|v_1|+|v_2| \le |v_1| \hspace{2pt} |v_2|$ for $|v_1|,|v_2|\ge2$.\\
Case 3: $| v_1 | \ge 2M$ and $|v_2| \le 2M$ we can use the results from Case 1 and Case 2 with the intermediate value $2M$ to obtain
\begin{align*}
|g_M(|v_1|^2) - g_M(|v_2|^2)| 
&\le |g_M(|v_1|^2) - g_M(|2M|^2)| + |g_M(|2M|^2) - g_M(|v_2|^2)| \\
&\le 4 M \gamma^{M,1} \left| |v_1| - |2M| \right| + \Gamma_M(M^2) \left| |2M| - |v_2 | \right| \\
&\le \max\{ 4 M \gamma^{M,1} ,  \Gamma_M(M^2)  \} \left( |v_1| - 2M + 2M - |v_2 | \right)\\
&\le \max\{ 4 M \gamma^{M,1} ,  \Gamma_M(M^2)  \}  |v_1 - v_2 |.
\end{align*}
\end{proof}

\begin{lemma}\label{lemma-L2-hatuhm-hatu}
Suppose $\partial_t u \in L^4(\mathcal{D} \times (0,T))$, $\partial_{tt} u \in L^2(\mathcal{D} \times (0,T))$ and $u \in W^{s,2}(0,T;H^2(\mathcal{D}))$ and  for $s=1$ or $s=2$.
Let $\hatuhM{n} \in \Sh$ denote a solution of the fully-discrete Crank-Nicolson Method with truncation as stated in Definition \ref{truncated-fully-discrete-crank-nic-gpe} and let $\deltat{}$ be small enough for the results of Theorem \ref{main-theorem-1} to hold. If $\hatu^{n}\in H^1_0(\mathcal{D})$ denotes the unique solution to \eqref{semi-disc-cnd-problem} with the properties stated in Theorem \ref{main-theorem-1}, then 
\begin{align*}
\| \hatuhM{n} - \PR(\hatu^{n}) \|_{L^2(\mathcal{D})} \le C(u,\gamma,V,\mathcal{D},M,\PR) h^2
\deltat{}^{s-2}
\end{align*}
holds with an $h$-independent constant $C(u,\gamma,\mathcal{D},M,\PR)$.
\end{lemma}

\begin{proof}
First, observe that the assumptions imply $\hatu^{n}=\hatuM{n}$ (and that it is unique). In the following, we denote by $C_M$ any generic constant that depends on $M$, $u$, $\gamma$, $V$ and $\mathcal{D}$. Recall the definition of the continuous function $\GM$ from \eqref{def-GM} and let again $(v,w)_{H^1_V}:=\Ltwo{\nabla v }{\nabla w} +\Ltwo{V \hspace{2pt} v }{ w}$. Consider $v\in\Sh$. From \eqref{fully-disc-cnd-problem-truncation} we have 
\begin{eqnarray*}
\lefteqn{\Ltwo{ \hatuhM{n} - \hatuhM{n-1}}{v} +
\frac{\deltat{n}}{2} \hspace{2pt}\ci \hspace{2pt}( \hatuhM{n} + \hatuhM{n-1} , v )_{H^1_V}}
\\
 &=& - \deltat{n} \hspace{2pt}\ci \hspace{2pt} \frac{1}{2} \Ltwo{ \GM(|\hatuhM{n}|^2,|\hatuhM{n-1}|^2) (\hatuhM{n} + \hatuhM{n-1}) }{v}.
\end{eqnarray*}
and from \eqref{semi-disc-cnd-problem}
\begin{eqnarray*}
\lefteqn{\Ltwo{ \PR(\hatu^{n}) - \PR(\hatu^{n-1}) }{v} +
\frac{\deltat{n}}{2} \hspace{2pt}\ci \hspace{2pt}
( \PR(\hatu^{n}+ \hatu^{n-1}) ,  v )_{H^1_V} }
\\
 &=& 
- \frac{\deltat{n}}{2} \hspace{2pt}\ci \hspace{2pt}
\Ltwo{ V \left( (\hatu^{n}+ \hatu^{n-1}) - \PR (\hatu^{n}+ \hatu^{n-1}) \right)}{ v}
\\
&\enspace& \quad + 
 \Ltwo{ \PR(\hatu^{n}) - \hatu^{n} - \PR(\hatu^{n-1}) + \hatu^{n-1}}{v}
 - \deltat{n} \hspace{2pt}\ci \hspace{2pt} \frac{1}{2} \Ltwo{ \GM(|\hatu^{n}|^2,|\hatu^{n-1}|^2) (\hatu^{n} + \hatu^{n-1}) }{v}.
\end{eqnarray*}
Subtracting the terms from each other and defining $\hatehM{n}:=\hatuhM{n} - \PR(\hatu^{n})$ gives us
\begin{eqnarray}
\nonumber\lefteqn{\Ltwo{ \hatehM{n} - \hatehM{n-1}}{v } +
 {\frac{\deltat{n}}{2}} \ci \hspace{2pt} 
 (\hatehM{n} + \hatehM{n-1} , v)_{H^1_V}}
 \\
\nonumber &=& 
\frac{\deltat{n}}{2} \hspace{2pt}\ci \hspace{2pt}
\Ltwo{ V \left( (\hatu^{n}+ \hatu^{n-1}) - \PR (\hatu^{n}+ \hatu^{n-1}) \right)}{ v}
- \Ltwo{ \PR(\hatu^{n}) - \hatu^{n} + \hatu^{n-1} - \PR(\hatu^{n-1}) }{v} \\
\nonumber&\enspace& \quad + \deltat{n} \hspace{2pt}\ci \hspace{2pt} \frac{1}{2} \Ltwo{ \GM(|\hatu^{n}|^2,|\hatu^{n-1}|^2) (\hatu^{n} + \hatu^{n-1}) }{v}\\
\nonumber&\enspace& \quad - \deltat{n} \hspace{2pt}\ci \hspace{2pt} \frac{1}{2} \Ltwo{ \GM(|\hatuhM{n}|^2,|\hatuhM{n-1}|^2) (\hatuhM{n} + \hatuhM{n-1}) }{v} \\
\nonumber &=& 
\frac{\deltat{n}}{2} \hspace{2pt}\ci \hspace{2pt}
\Ltwo{ V \left( (\hatu^{n}+ \hatu^{n-1}) - \PR (\hatu^{n}+ \hatu^{n-1}) \right)}{ v}
- \Ltwo{ \PR(\hatu^{n}) - \hatu^{n} + \hatu^{n-1} - \PR(\hatu^{n-1}) }{v} \\
\nonumber&\enspace& \quad + \deltat{n} \hspace{2pt}\ci \hspace{2pt} \frac{1}{2} \Ltwo{ \GM(|\hatu^{n}|^2,|\hatu^{n-1}|^2) (\hatu^{n} + \hatu^{n-1}) }{v}\\
\nonumber&\enspace& \quad - \deltat{n} \hspace{2pt}\ci \hspace{2pt} \frac{1}{2} \Ltwo{ \GM(|\hatuhM{n}|^2,|\hatuhM{n-1}|^2) (\hatehM{n} + \hatehM{n-1}) }{v} \\
\nonumber&\enspace& \quad - \deltat{n} \hspace{2pt}\ci \hspace{2pt} \frac{1}{2} \Ltwo{ \GM(|\hatuhM{n}|^2,|\hatuhM{n-1}|^2) ( \PR(\hatu^{n}) - \hatu^{n} + \PR(\hatu^{n-1}) - \hatu^{n-1}) }{v} \\
\nonumber&\enspace& \quad - \deltat{n} \hspace{2pt}\ci \hspace{2pt} \frac{1}{2} \Ltwo{ \GM(|\hatuhM{n}|^2,|\hatuhM{n-1}|^2) (  \hatu^{n} + \hatu^{n-1}) ) }{v}.
\end{eqnarray}
Testing with $v= \hatehM{n} + \hatehM{n-1} $ and taking the real part yields
\begin{eqnarray}
\nonumber\lefteqn{ \| \hatehM{n} \|_{L^2(\mathcal{D})}^2 - \| \hatehM{n-1} \|_{L^2(\mathcal{D})}^2 }\\
\nonumber &=& 
\underset{=:\mbox{I}}{\underbrace{- \frac{\deltat{n}}{2} \hspace{2pt}
\Im \Ltwo{ V \left( (\hatu^{n}+ \hatu^{n-1}) - \PR (\hatu^{n}+ \hatu^{n-1}) \right)}{  \hatehM{n} + \hatehM{n-1} }}}
\\
\nonumber&\enspace& \quad 
\underset{=:\mbox{II}}{\underbrace{-\Re \Ltwo{ \PR(\hatu^{n}) - \hatu^{n} - ( \PR(\hatu^{n-1}) - \hatu^{n-1} )}{  \hatehM{n} + \hatehM{n-1} }}} \\
\nonumber&\enspace& \quad + \underset{=:\mbox{III}}{\underbrace{ \frac{\deltat{n}}{2} \Im \Ltwo{ \GM(|\hatuhM{n}|^2,|\hatuhM{n-1}|^2) ( \PR(\hatu^{n}) - \hatu^{n} + \PR(\hatu^{n-1}) - \hatu^{n-1}) }{ \hatehM{n} + \hatehM{n-1} }}} \\
\nonumber&\enspace& \quad - \underset{=:\mbox{IV}}{\underbrace{\frac{\deltat{n}}{2} \Im \Ltwo{ \left(\GM(|\hatu^{n}|^2,|\hatu^{n-1}|^2) - \GM(|\hatuhM{n}|^2,|\hatuhM{n-1}|^2) \right) (  \hatu^{n} + \hatu^{n-1}) ) }{ \hatehM{n} + \hatehM{n-1} }}}.
\end{eqnarray}
For the first term we have with Theorem \ref{main-theorem-1}
\begin{align*}
\mbox{I} &\le \frac{\deltat{n}}{2} \| V \|_{L^{\infty}(\mathcal{D})} C_{\PR}  | h^{2} (\hatu^{n} + \hatu^{n-1}) |_{H^2(\mathcal{D})} \| \hatehM{n}+ \hatehM{n-1} \|_{L^2(\mathcal{D})}\\
&\le C_M \deltat{n} h^4 + \deltat{n} \left( \| \hatehM{n} \|_{L^2(\mathcal{D})}^2 + \| \hatehM{n-1} \|_{L^2(\mathcal{D})}^2 \right).
\end{align*}
The second term can be estimates as
\begin{align*}
\mbox{II} &\le C_{\PR}  | h^{2} (\hatu^{n} - \hatu^{n-1}) |_{H^2(\mathcal{D})} \| \hatehM{n}+ \hatehM{n-1} \|_{L^2(\mathcal{D})}\\
&\le C_M h^{2} \left( | \hatu^{n} - u^n |_{H^2(\mathcal{D})} +  | \hatu^{n-1} - u^{n-1} |_{H^2(\mathcal{D})}  
+  | \int_{I_n} \partial_t u |_{H^2(\mathcal{D})} \right) \| \hatehM{n}+ \hatehM{n-1} \|_{L^2(\mathcal{D})} \\
&\le C_M h^{2} \left( \deltat{}^{{s-1}} + \sqrt{\deltat{}} | \partial_t u |_{L^2(I_n,H^2(\mathcal{D}))} \right) \| \hatehM{n}+ \hatehM{n-1} \|_{L^2(\mathcal{D})}\\
&\le C_M h^{4} \left( \deltat{}^{2s-4} \deltat{} + | \partial_t u |_{L^2(I_n,H^2(\mathcal{D}))}^2 \right)
+ C_M \deltat{} \left( \| \hatehM{n} \|_{L^2(\mathcal{D})}^2 + \| \hatehM{n-1} \|_{L^2(\mathcal{D})}^2 \right),
\end{align*}
where we used again Theorem \ref{main-theorem-1}. 
For the third term we can proceed analogously since $G_M$ is bounded. We obtain straightforwardly (again with Theorem \ref{main-theorem-1}) that
\begin{align*}
\mbox{III} &\le C_M \deltat{} h^{2} \left( | \hatu^{n}  |_{H^2(\mathcal{D})} + | \hatu^{n-1} |_{H^2(\mathcal{D})} \right) \| \hatehM{n}+ \hatehM{n-1} \|_{L^2(\mathcal{D})}\\
&\le C_M \deltat{} h^{4}+ C_M \deltat{} \left( \| \hatehM{n} \|_{L^2(\mathcal{D})}^2 + \| \hatehM{n-1} \|_{L^2(\mathcal{D})}^2 \right).
\end{align*}
To bound term IV, we use Lemma \ref{truncated-fully-discrete-crank-nic-gpe} to estimate
\begin{eqnarray*}
\lefteqn{\left| \GM(|\hatu^{n}|^2,|\hatu^{n-1}|^2) - \GM(|\hatuhM{n}|^2,|\hatuhM{n-1}|^2) \right|} \\
&\le& \left| \GM(|\hatu^{n}|^2,|\hatu^{n-1}|^2) - \GM(|\hatu^{n}|^2,|\hatuhM{n-1}|^2) \right|
+ \left| \GM(|\hatu^{n}|^2,|\hatuhM{n-1}|^2) - \GM(|\hatuhM{n}|^2,|\hatuhM{n-1}|^2) \right| \\
&\le& C_M \left| \hatu^{n-1} - \hatuhM{n-1} \right| + C_M \left| \hatu^{n} - \hatuhM{n} \right|\\
&\le& C_M \left( |\hatehM{n}| + |\hatehM{n-1}| +  \left| \hatu^{n-1} - \PR( \hatu^{n-1}) \right| + \left| \hatu^{n} - \PR(\hatu^{n}) \right| \right).
\end{eqnarray*}
Consequently, using that $\| \hatu^{n} \|_{L^{\infty}(\mathcal{D})}$ is uniformly bounded (Theorem \ref{main-theorem-1}) we can conclude
\begin{align*}
\mbox{IV} &\le C_M \deltat{} h^{4}+ C_M \deltat{} \left( \| \hatehM{n} \|_{L^2(\mathcal{D})}^2 + \| \hatehM{n-1} \|_{L^2(\mathcal{D})}^2 \right).
\end{align*}
Collecting the estimates for I, II, III and IV implies that 
\begin{eqnarray*}
(1 - C_M \deltat{n}) \| \hatehM{n} \|_{L^2(\mathcal{D})}^2 \le 
(1 + C_M \deltat{n}) \| \hatehM{n-1} \|_{L^2(\mathcal{D})}^2 + C_M \deltat{} h^{4} \deltat{}^{2s-4} + h^{4} | \partial_t u |_{L^2(I_n,H^2(\mathcal{D}))}^2.
\end{eqnarray*}
Using the inequality $a_{n+1} \le e^{\sum_{\ell=0}^n \alpha_\ell} \left( a_0 + \sum_{\ell=0}^n b_\ell \right)$ which holds for any $0\le a_n,b_n,\alpha_n$ with $a_{n+1} \le (1 + \alpha_n) a_n + b_n$ finishes the $L^2$-error estimate.
\end{proof}

We can now conclude from Lemma \ref{lemma-L2-hatuhm-hatu} that $\hatuhM{n}$ remains uniformly bounded in $L^{\infty}$ which allows us to conclude $\hatuhM{n}=\hatuh^{n}$ for appropriately chosen $M$. In summary we obtain Theorem \ref{main-theorem-2-a}. The detailed proof is given in the following.

\begin{proof}[Proof of Theorem \ref{main-theorem-2-a} and \ref{main-theorem-2-b}]
We choose $M:=1+ 2 \| u \|_{L^{\infty}(\mathcal{D} \times (0,T))} 
+ \sup_{0\le n \le N} \| u^n \|_{H^{2}(\mathcal{D})}$.
Let $C=C(u,\gamma,V,\mathcal{D},\PR)$ and let $C_M$ denote any constant depending additionally on $M$ (however, both are not allowed to depend on $h$ or $\deltat{}$). Using the assumptions on $\PR$, the bounds from Theorem \ref{main-theorem-1} and Lemma \ref{lemma-L2-hatuhm-hatu} we have
\begin{eqnarray*}
\|\hatuhM{n} \|_{L^{\infty}(\mathcal{D})}
&\le& 
\|\PR(\hatu^{n}) \|_{L^{\infty}(\mathcal{D})}
+ \| \hatuhM{n} - \PR(\hatu^{n}) \|_{L^{\infty}(\mathcal{D})}\\
&\overset{\eqref{inverse-estimate-PR},\eqref{Linfty-stability-PR}}{\le}& \| \hatu^{n}  \|_{H^{2}(\mathcal{D})} + C_{\mbox{\rm\tiny inv}} h^{2-d/2} C_M \\
&\le& 2 \| u \|_{L^{\infty}(\mathcal{D} \times (0,T))} 
+ \sup_{0\le n \le N} \| u^n \|_{H^{2}(\mathcal{D})} + h^{2-d/2}\deltat{}^{s-2} C_M. 
\end{eqnarray*}
Since $d=2,3$ and $h^{4-d-\alpha} \lesssim \deltat{}^2$ for $s=1$ and some $\alpha>0$,
we conclude that there exists $\hat{h}>0$ such that
\begin{align*}
\| \hatuhM{n} \|_{L^{\infty}(\mathcal{D})} \le M
\end{align*}
for all $h < \hat{h}$. Hence, for sufficiently small $h$ we have $\hatuhM{n}=\hatuh^{n}$. We conclude the existence of $\hatuh^{n}$ and the $h$-independent bound
\begin{align*}
\| \hatuh^{n} \|_{L^{\infty}(\mathcal{D})} \le 1+ 2 \| u \|_{L^{\infty}(\mathcal{D} \times (0,T))} 
+ \sup_{0\le n \le N} \| u^n \|_{H^{2}(\mathcal{D})}.
\end{align*}
For the $L^2$-error estimate we split the error into
\begin{align*}
\| u^n - \hatuh^{n} \|_{L^{2}(\mathcal{D})} \le 
\| u^n - \hatu^{n} \|_{L^{2}(\mathcal{D})}
+ \| \hatu^{n} - \PR(\hatu^{n}) \|_{L^{2}(\mathcal{D})}
+ \| \PR(\hatu^{n}) - \hatuh^{n} \|_{L^{2}(\mathcal{D})}.
\end{align*}
The first term can be estimated with Theorem \ref{main-theorem-1} for sufficiently small $\deltat{}$, the second term is bounded by $\| \hatu^{n} - \PR(\hatu^{n}) \|_{L^{2}(\mathcal{D})} \le C h^2 \|  \hatu^{n}  \|_{H^{2}(\mathcal{D})} \le C h^2$ (again using Theorem \ref{main-theorem-1}) and the last term can be estimated with Lemma \ref{lemma-L2-hatuhm-hatu} which now holds with $\hatuhM{n}=\hatuh^{n}$. In the setting of Theorem \ref{main-theorem-2-a}, this yields
\begin{align*}
\| u^n - \hatuh^{n} \|_{L^{2}(\mathcal{D})} \le C(h^2 + \deltat{}^2)
\end{align*}
for all sufficiently small $h$ and $\deltat{}$. In the setting of Theorem \ref{main-theorem-2-b}, the order is reduced to $h^{(d+\alpha)/2} + \deltat{}$.
The proof of uniqueness under some uniform bound $\|\hatuhM{n} \|_{L^{\infty}(\mathcal{D})}\le C$ independent of $\deltat{}$ and $h$ is almost verbatim the same as in the semi-discrete case (see the proof of Theorem \ref{main-theorem-1}).
\end{proof}

\section{Numerical experiment}\label{sec:numexp}
We shall conclude with some simple and illustrative numerical experiment. The computational domain is given by $\mathcal{D}:=[-6,6]^2$ and the time interval by $[0,T]:=[0,1]$. We wish to approximate $u: \mathcal{D} \times [0,T] \rightarrow \C$ with $u(x,0)=u^0$ and
\begin{align*}
\ci \partial_t u &= - \frac{1}{2} \triangle u + Vu + \beta |u|^2 u  \qquad \mbox{in } \mathcal{D} \times [0,T] , \\
u &= 0 \hspace{121pt} \mbox{on } \partial \mathcal{D} \times [0,T].
\end{align*}
Here we have $\beta=20$ and 
and we consider a rough potential $V\in L^{\infty}(\mathcal{D})$ given by
\begin{align*}
V(x) = \mbox{\rm int} \left( 5  + 2 \sin( \frac{\pi x_1}{3} ) \sin( \frac{\pi x_2}{3} ) \right),
\end{align*}
where \quotes{$\mbox{\rm int}$} rounds a real number $r$ to the largest integer smaller or equal to $r$. The potential is visualized in Figure \ref{potential}. 
We note that the potential $V$ is not a confinement potential as it does not fulfill $V(x)\rightarrow \infty$ for $|x|\rightarrow \infty$. For that reason, the physically correct solution will escape $\mathcal{D}$ for sufficiently large times $t$. In our experiment we picked the maximum time $T=1$ small enough so that this does not happen.
\begin{figure}[h!]
\centering
\includegraphics[scale=0.35]{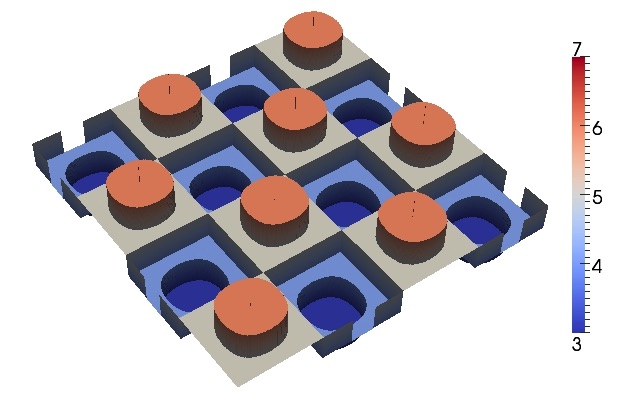}
\caption{\it Discontinuous potential $V(x)$ as considered in the model problem.}
\label{potential}
\end{figure}
Inspired by the discussion in Appendix \ref{appendix-b} we select the initial value as the ground state of a perturbed NLS. More precisely, we choose $u^0 \in H^1_0(\mathcal{D})$ with $u^0\ge 0$ such that 
\begin{align*}
E(u^0) = \underset{\| v \|_{L^2(\mathcal{D})}=1}{\min_{v \in H^1_0(\mathcal{D})}} E(v), \qquad \mbox{where  } \quad
E(v):= \frac{1}{2} \int_{\mathcal{D}} |\nabla v|^2 + \int_{\mathcal{D}} (V + V_{s}) \hspace{2pt} |v|^2 + \frac{\beta_s}{2} \int_{\mathcal{D}}  |v|^4,
\end{align*}
with $\beta_s=10$ and a smooth potential perturbation $V_s(x):=\frac{1}{2}(x_1^2+ x_2^2)$. There exists a unique ground state $u^0$ with the above properties (cf. \cite{CCM10}) and it holds $u^0 \in H^2(\mathcal{D})$. Given a finite element space $S_h$, the discrete approximation of $u^0$ in $S_h$ is given by some $u_{h}^0\in S_h$ with $\| u_{h}^0 \|_{L^2(\mathcal{D})}=1$ and
\begin{align}
\label{disc-initial-value}
E(u_{h}^0) = \underset{\| v \|_{L^2(\mathcal{D})}=1}{\min_{v \in S_h}} E(v),
\end{align}
i.e., $u_{h}^0$ is an energy minimizer in $S_h$. Such a minimizer exists and it holds $\| u^0 - u_{h}^0 \|_{L^2(\mathcal{D})} \le C(u_0) h^2$ (independent of the smoothness of $V$). This means that using $u_{h}^0$ as a discrete initial value in our scheme \eqref{crank-nic-gpe} will introduce an error that is of the same order as if using $\PR(u^0)$. We compute the discrete minimizers $u_{h}^0$ by using the Discrete Normalized Gradient Flow method proposed in \cite{BaD04}. For $S_{h}$ we use a Lagrange finite element space of polynomial order $1$, based on a uniform (simplicial) triangulation of $\mathcal{D}$. The mesh size $h$ is given as the diameter of the elements of the triangulation. For $h=2^{-6} \sqrt{2} \hspace{2pt} \mbox{\rm diam}(\mathcal{D})$ the discrete ground state $u_{h,0}$ is depicted in Figure \ref{reference-values}. In the following, all errors are with respect to a reference solution $u_{\mbox{\rm\tiny ref}}^{n}:=u_{h,\deltat{}}^{n} \in \Sh$ computed with the Crank-Nicolson scheme \eqref{cnd-problem} with $h=2^{-6} \sqrt{2} \hspace{2pt} \mbox{\rm diam}(\mathcal{D})$ and with equidistant time steps of size $\deltat{}=10^{-2}$. The reference solution $u_{\mbox{\rm\tiny ref}}^{n}$ at $T=1$ is depicted in the right picture of Figure \ref{reference-values}.
\begin{figure}[h!]
\centering
\includegraphics[scale=0.35]{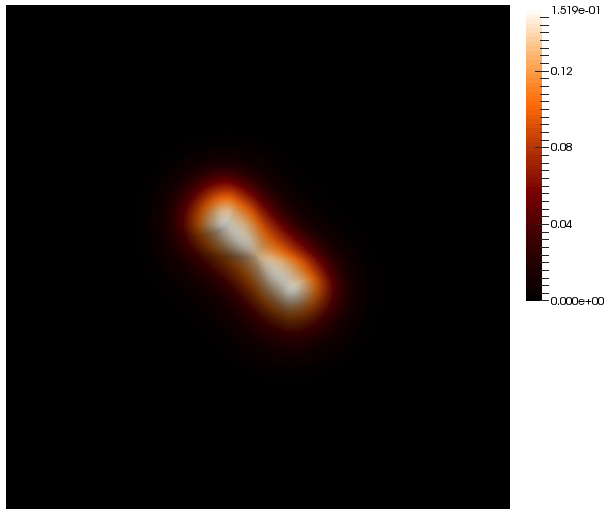}
\includegraphics[scale=0.35]{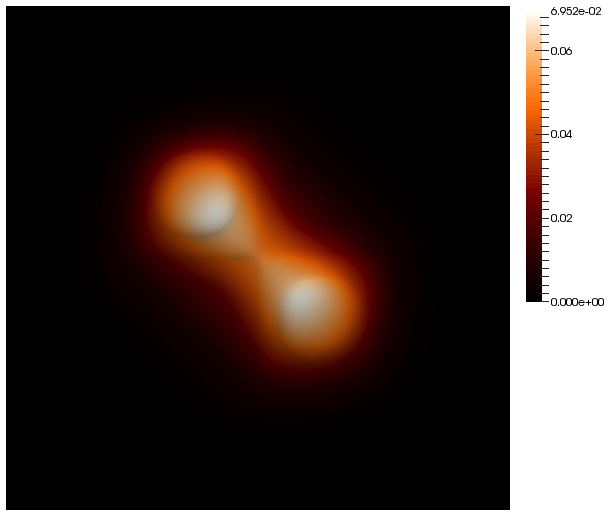}
\caption{\it Left: reference initial value $u_{h}^0$ obtained for $h^{\mbox{\tiny rel}}=2^{-6}$. Right: reference solution $u_{\mbox{\rm\tiny ref}}^{n}$ at $T=1$ obtained with \eqref{cnd-problem} for $h^{\mbox{\tiny rel}}=2^{-6}$ and $\deltat{}=10^{-2}$.}
\label{reference-values}
\end{figure}
In the following with present the discrete approximations $u_{h,\deltat{}}^n$ obtained with \eqref{cnd-problem} in $\Sh$ and with equidistant time steps. Recall that the discrete initial value is given by \eqref{disc-initial-value}, i.e. $u_{h,\deltat{}}^0=u_{h}^0$.
 
Before discussing the error evolution, we introduce some short-hand notation. The rescaled mesh and time step sizes are given by
\begin{align*}
h^{\mbox{\tiny rel}} := (\sqrt{2} \hspace{2pt} \mbox{\rm diam}(\mathcal{D}))^{-1} h
\qquad \mbox{and} \qquad
\deltat{}^{\mbox{\tiny rel}}:=3/2 \deltat{}.
\end{align*}
The error compared to the reference solution a time $t^n$ is denoted by
$$
e^{n}_{h,\deltat{}} := u_{\mbox{\rm\tiny ref}}^{n} - u_{h,\deltat{}}^n
$$
and the corresponding relative errors (for real and imaginary parts) are given by
\begin{align*}
\| \Re \hspace{2pt} e^{N}_{h,\deltat{}} \|_{L^2(\mathcal{D})}^{\mbox{\tiny rel}} := \frac{\| \Re \hspace{2pt} e^{N}_{h,\deltat{}} \|_{L^2(\mathcal{D})}}{\| \Re \hspace{2pt} u_{\mbox{\rm\tiny ref}}^{n} \|_{L^2(\mathcal{D})}} \qquad \mbox{and} \qquad
\| \Im \hspace{2pt} e^{N}_{h,\deltat{}} \|_{L^2(\mathcal{D})}^{\mbox{\tiny rel}} := \frac{\| \Im \hspace{2pt} e^{N}_{h,\deltat{}} \|_{L^2(\mathcal{D})}}{\| \Im \hspace{2pt} u_{\mbox{\rm\tiny ref}}^{n} \|_{L^2(\mathcal{D})}}
\end{align*}
and analogously for the error in the gradient. The EOCs in Tables \ref{table-eocs-coupled}, \ref{table-lerrors-new-1} and \ref{table-lerrors-new-2} refer to the averages of the (E)xperimental (O)rders of (C)onvergence.

\begin{table}[h!]
\caption{\it The table shows errors between the reference solution $u_{\mbox{\rm\tiny ref}}^{n}$ and various Crank-Nicolson approximations $u_{h,\deltat{}}^n$ for simultaneously refined spatial and temporal meshes.}
\label{table-eocs-coupled}
\begin{center}
\begin{tabular}{|c|c|c|c|c|c|c|c|}
\hline $h^{\mbox{\tiny rel}}$      & $\deltat{}^{\mbox{\tiny rel}}$
& $\| \Re \hspace{2pt} e^{N}_{h,\deltat{}} \|_{L^2(\mathcal{D})}^{\mbox{\tiny rel}}$
& $\| \Im \hspace{2pt} e^{N}_{h,\deltat{}} \|_{L^2(\mathcal{D})}^{\mbox{\tiny rel}}$
& $\| \Re \hspace{2pt} \nabla e^{N}_{h,\deltat{}} \|_{L^2(\mathcal{D})}^{\mbox{\tiny rel}}$
& $\| \Im \hspace{2pt} \nabla e^{N}_{h,\deltat{}} \|_{L^2(\mathcal{D})}^{\mbox{\tiny rel}}$ \\
\hline
\hline $2^{-2}$ & $2^{-2} $  & 0.7157 & 0.7603 & 0.9929 & 0.8506 \\
\hline $2^{-3}$ & $2^{-3} $  & 0.1753 & 0.2370 & 0.4045 & 0.4379 \\
\hline $2^{-4}$ & $2^{-4} $  & 0.0236 & 0.0338 & 0.0881 & 0.0935 \\
\hline $2^{-5}$ & $2^{-5} $  & 0.0050 & 0.0069 & 0.0205 & 0.0217 \\
\hline \multicolumn{2}{|c|}{EOC} & 2.38 & 2.26 & 1.86 & 1.76 \\
\hline
\end{tabular}\end{center}
\end{table}

\begin{table}[h!]
\caption{\it The table shows $L^{\infty}(L^2)$- and $L^{\infty}(H^1)$-errors between the reference solution $u_{\mbox{\rm\tiny ref}}^{n}$ and various Crank-Nicolson approximations $u_{h,\deltat{}}^n$ obtained with \eqref{cnd-problem} for fixed $\deltat{}^{\mbox{\tiny rel}}=2^{-6}$ and decreasing mesh sizes.}
\label{table-lerrors-new-1}
\begin{center}
\begin{tabular}{|c|c|c|c|c|c|c|c|}
\hline $h^{\mbox{\tiny rel}}$      
& $\| \Re \hspace{2pt} e^{N}_{h,\deltat{}} \|_{L^2(\mathcal{D})}^{\mbox{\tiny rel}}$
& $\| \Im \hspace{2pt} e^{N}_{h,\deltat{}} \|_{L^2(\mathcal{D})}^{\mbox{\tiny rel}}$
& $\| \Re \hspace{2pt} \nabla e^{N}_{h,\deltat{}} \|_{L^2(\mathcal{D})}^{\mbox{\tiny rel}}$
& $\| \Im \hspace{2pt} \nabla e^{N}_{h,\deltat{}} \|_{L^2(\mathcal{D})}^{\mbox{\tiny rel}}$ \\
\hline
\hline $2^{-2}$ & 0.5571 & 1.1710 & 0.7954 & 1.1367 \\
\hline $2^{-3}$ & 0.2063 & 0.2415 & 0.4562 & 0.4780 \\
\hline $2^{-4}$ & 0.0259 & 0.0297 & 0.1006 & 0.1061 \\
\hline $2^{-5}$ & 0.0015 & 0.0017 & 0.0195 & 0.0206 \\
\hline
\hline 
 EOC & 2.85 & 3.15  & 1.78  & 1.93 \\
\hline
\end{tabular}\end{center}
\end{table}

\begin{table}[h!]
\caption{\it The table shows $L^{\infty}(L^2)$- and $L^{\infty}(H^1)$-errors between the reference solution $u_{\mbox{\rm\tiny ref}}^{n}$ and various Crank-Nicolson approximations $u_{h,\deltat{}}^n$ obtained with \eqref{cnd-problem} 
or fixed $h^{\mbox{\tiny rel}}=2^{-5}$ and decreasing time step sizes.}
\label{table-lerrors-new-2}
\begin{center}
\begin{tabular}{|c|c|c|c|c|c|c|c|}
\hline $\deltat{}^{\mbox{\tiny rel}}$
& $\| \Re \hspace{2pt} e^{N}_{h,\deltat{}} \|_{L^2(\mathcal{D})}^{\mbox{\tiny rel}}$
& $\| \Im \hspace{2pt} e^{N}_{h,\deltat{}} \|_{L^2(\mathcal{D})}^{\mbox{\tiny rel}}$
& $\| \Re \hspace{2pt} \nabla e^{N}_{h,\deltat{}} \|_{L^2(\mathcal{D})}^{\mbox{\tiny rel}}$
& $\| \Im \hspace{2pt} \nabla e^{N}_{h,\deltat{}} \|_{L^2(\mathcal{D})}^{\mbox{\tiny rel}}$ \\
\hline
\hline $2^{-2} $  & 0.3629 & 0.5156 & 0.5020 & 0.5665 \\
\hline $2^{-3} $  & 0.1088 & 0.1451 & 0.1696 & 0.1832 \\
\hline $2^{-4} $  & 0.0269 & 0.0356 & 0.0471 & 0.0506 \\
\hline $2^{-5} $  & 0.0050 & 0.0069 & 0.0205 & 0.0217 \\
\hline $2^{-6} $  & 0.0015 & 0.0017 & 0.0195 & 0.0206 \\
\hline
\hline 
 EOC & 1.98 & 2.06  & 1.17  & 1.20 \\
\hline
\end{tabular}\end{center}
\end{table}

In order to study the accuracy of the Crank-Nicolson finite element method stated in \eqref{cnd-problem}, we run various computations with different constellations for the size of the mesh size $h$ and the time step size $\deltat{}$. As there is no known exact solution to our model problem, we use the fine scale approximation $u_{\mbox{\rm\tiny ref}}^{n}$ as our reference for the computation of errors. 
The results of the computations are depicted in Tables \ref{table-eocs-coupled}, \ref{table-lerrors-new-1} and \ref{table-lerrors-new-2}.
From that we can make several observations. First, we observe a clearly convergent behavior in terms of the mesh size and time step size and we did not encounter any numerical issues (on the solver level) when computing the approximations $u_{h,\deltat{}}^n$. We can also report that the scheme preserved the mass and the energy almost up to machine precision. Second, the computed experimental orders of convergence do not correlate with the pessimistic linear rates predicted by Theorem \ref{main-theorem-2-b}. More precisely, we rather observe the quadratic rates expected under the stronger regularity assumptions of Theorem \ref{main-theorem-2-a}. 
This is emphasized by Table \ref{table-eocs-coupled}, where we depict the 
EOCs for the case that $h$ and $\deltat{}$ are refined simultaneously. 
In Table \ref{table-lerrors-new-1} we fix the time step size  $\deltat{}^{\mbox{\tiny rel}}=2^{-6}$ and only refine the spatial mesh. The convergence in terms of $h$ seems to be almost cubic for the $L^2$-error and almost quadratic for the $H^1$-error. In Table \ref{table-lerrors-new-2} we fix the mesh size with $h^{\mbox{\tiny rel}}=2^{-5}$, whereas the time steps become smaller. Here we observe a roughly quadratic rate in $\tau$ for the $L^2$-error and a linear rate for the $H^1$-error.
In the light of these results, the performance of the method appears better than predicted. This might indicate that the regularity assumption (R3) is feasible also for some class of discontinuous potentials. However, an empirical proof or disproof of this claim requires further systematic numerical studies beyond the scope of this paper. In particular, we cannot exclude the possibility of super-convergence effects when estimating the error using some reference solution, instead of the unknown exact solution.

$\\$
{\bf Conclusion.} In this paper we analyzed a mass- and energy conserving Crank-Nicolson Galerkin method. We showed that it is numerically stable under perturbations, that the scheme is well-posed in some ball (in $L^{\infty}(\mathcal{D}\times [0,T])$) around zero and we derived $L^{\infty}(L^2)$-error estimates under various regularity assumptions. All our estimates are valid for general disorder potentials in $L^{\infty}(\mathcal{D})$. However, it is not clear how or if our regularity assumptions might conflict with discontinuities in the potential. Therefore we derived two graded results. In the first main result, we assume sufficient regularity of the exact solution $u$ and derive error estimates of optimal (quadratic) order in $h$ and $\deltat{}$. The novelty with respect to previous works is that our results cover a general class of nonlinearities, potential terms and we show that the method does indeed not require a time step constraint. On the contrary, the results in \cite{San84,ADK91} are only valid, provided that the time step size is sufficiently small with respect to the spatial mesh size. In our second main result, we relax the regularity assumptions so that they appear not to be in conflict with discontinuous potentials. Under these relaxed regularity assumptions, we can still derive $L^{\infty}(L^2)$-error estimates, however, only of linear order. Furthermore, we encounter a time step constraint that was absent in the case of higher regularity. To check the practical performance of the method, we present a numerical experiment for a model problem with discontinuous potential. The corresponding numerical errors seem not to correlate with the pessimistic rates predicted for the low-regularity regime. We could neither observe degenerate convergence rates nor a practical time step constraint. Instead, we observe the behavior as predicted for the high regularity regime, i.e., convergence rates of optimal order and good approximations in all resolution regimes, independent of a coupling between mesh size and time step size.

\def\cprime{$'$}

\begin{appendix}

\section{Compatibility of regularity assumptions and discontinuous potentials}
\label{appendix-b}
In this appendix, we demonstrate that discontinuous potentials and the regularity assumptions (R1) and (R2) are actually compatible for proper initial values $u_0$. For simplicity of the presentation we consider the linear case, i.e., $\gamma=0$. The nonlinear case is briefly discussed at the end.
 
Let $V_d \in L^{\infty}(\mathcal{D})$ denote a rough disorder potential and let $u_0 \in H^1_0(\mathcal{D})$ denote a ground state or excited state to the 
stationary
Schr\"odinger equation
\begin{align}
\label{eigenvalue-prob}- \triangle u_0 + V_d u_0 
= \lambda u_0,
\end{align}
where $\lambda>0$ is the corresponding eigenvalue (the chemical potential) and $u_0$ is $L^2$-normalized, i.e. $\| u_0 \|_{L^2(\mathcal{D})}=1$. From elliptic regularity theory we know that the solution to problem \eqref{eigenvalue-prob} admits higher regularity, i.e. $u_0 \in C^{0}(\overline{\mathcal{D}}) \cap H^2(\mathcal{D})$ (cf. \cite{GiT01}). However, since $V_d$ is rough, we cannot expect any regularity beyond $H^2$.
In order to investigate the dynamics of $u_0$, the potential trap is reconfigured. In our case this means that we set $V:=V_d + V_s$, where $V_s$ is a non-negative smooth perturbation, say (for simplicity) $V_s \in C_0^{\infty}(\mathcal{D})$. With this we seek $u : [0,T] \rightarrow H^1_0(\mathcal{D})$ with $u(0)=u_0$ and
\begin{align}
\label{timedependent-prob-numexp} \ci \partial_t u = - \triangle u + V u. 
\end{align} 
Let us now assume that $u$ denotes a solution to \eqref{timedependent-prob-numexp} that is sufficiently regular. Then, from equation \eqref{timedependent-prob-numexp} we conclude that
\begin{align*}
\ci \Ltwo{ \partial_{t}^{k+1} u }{ \partial_t^{k} u } = \Ltwo{ \nabla \partial_t^{k} u}{ \nabla \partial_t^{k} u} + \Ltwo{ V \partial_t^{k} u }{ \partial_t^{k} u }.
\end{align*} 
for $k=0,1,2,3$. Taking only the imaginary part of the equation yields
\begin{align*}
0 = \Re \Ltwo{ \partial_{t}^{k+1} u }{ \partial_t^k u } = \frac{d}{dt} \|  \partial_t^k u \|_{L^2(\mathcal{D})}^2.
\end{align*}
By integrating from $0$ to $t\le T$, we have 
\begin{align*}
\|  \partial_t^{k} u(t) \|_{L^2(\mathcal{D})} = \|  \partial_t^{k} u(0) \|_{L^2(\mathcal{D})}.
\end{align*}
This means that we have to verify that the compatibility \quotes{$\partial_t^{k} u(0) \simeq \partial_t^{k} u_0$} is well-defined for rough potentials. For $k=1$ we exploit the initial condition and obtain
\begin{align}
\label{equality-dt-u-0}
\ci \partial_t u(0) 
= - \triangle u_0 + V_d u_0 + V_s u_0
\overset{\eqref{eigenvalue-prob}}{=}  \lambda u_0 + V_s u_0.
\end{align}
Hence,  
\begin{align*}
\|  \partial_t u \|_{L^{\infty}(0,T;L^2(\mathcal{D}))} \le \|  \lambda u_0 + V_s u_0 \|_{L^2(\mathcal{D})}.
\end{align*}
Analogously, we obtain for $k=2$ that
\begin{align}
\nonumber \ci \partial_{tt} u(0) 
&= - \triangle  \partial_{t} u_0 + V_d \partial_{t} u_0 + V_s \partial_{t} u_0
\overset{\eqref{equality-dt-u-0}}{=} 
 - \triangle \left( \lambda u_0 + V_s u_0 \right) + V_d \left( \lambda u_0 + V_s u_0 \right) + V_s \left( \lambda u_0 + V_s u_0 \right)\\
\nonumber &\overset{\eqref{eigenvalue-prob}}{=} 
 \lambda^2 u_0 
  - ( \triangle V_s) u_0 - 2 \nabla V_s \cdot \nabla u_0 
 - V_s \triangle u_0 + V_s V_d u_0 + V_s \left( \lambda u_0 + V_s u_0 \right) \\
\label{equality-dtt-u-0}  &\overset{\eqref{eigenvalue-prob}}{=} 
   \left( (V_s + \lambda)^2 - \triangle V_s \right) u_0 
 - 2 \nabla V_s \cdot \nabla u_0.
\end{align}
Hence 
\begin{align*}
\|  \partial_{tt} u \|_{L^{\infty}(0,T;L^2(\mathcal{D}))} \le \|  \left( (V_s + \lambda)^2 - \triangle V_s \right) u_0 
 - 2 \nabla V_s \cdot \nabla u_0 \|_{L^2(\mathcal{D})}.
\end{align*}
Furthermore, since 
\begin{align*}
- \triangle (\partial_{t} u) =  \ci \partial_{tt} u - V \partial_{t} u,
\end{align*} 
where $\ci \partial_{tt} u - V \partial_{t} u \in L^{\infty}(0,T;L^2(\mathcal{D}))$ (for which we just derived corresponding bounds depending on $u_0$), we can also conclude by elliptic regularity theory that
\begin{align*}
\| \partial_{t} u\|_{L^{\infty}(0,T;H^2(\mathcal{D}))} \le C(u_0).
\end{align*} 
The equation $\ci \partial_{tt} u(0)=\left( (V_s + \lambda)^2 - \triangle V_s \right) u_0 - 2 \nabla V_s \cdot \nabla u_0$ leads to another important observation: We have $\partial_{tt} u(0) \in H^1_0(\mathcal{D})$, but we do {\it not} have $H^2$-regularity as this would require $\nabla u_0 \in H^2(\mathcal{D})$, which is clearly not available due to the roughness of $V_d$. Therefore we cannot repeat the same argument for $\partial_{ttt} u(0)$. Observe that for $k=3$ we have
\begin{align*}
\ci \partial_{ttt} u(0) 
&\overset{\eqref{equality-dtt-u-0}}{=} 
- \triangle \left( \left( (V_s + \lambda)^2 - \triangle V_s \right) u_0 
 - 2 \nabla V_s \cdot \nabla u_0 \right) + V \left( \left( (V_s + \lambda)^2 - \triangle V_s \right) u_0 
 - 2 \nabla V_s \cdot \nabla u_0 \right),
\end{align*}
which implies that $\partial_{ttt} u(0)  \in H^{-1}(\mathcal{D})$, but it is not in $L^2(\mathcal{D})$. In order to obtain $\partial_{ttt} u  \in L^{\infty}(L^2)$ the potential needs to be at lest in $H^1$ (which contradicts the notion of a discontinuous or rough potential). Consequently, we can neither hope for $\partial_{ttt} u  \in L^{\infty}(L^2)$ nor for $\partial_{tt} u  \in L^{\infty}(H^2)$.
The only thing that we can hope for is to verify $\partial_{tt} u \in L^{\infty}(H^1)$. And indeed, analogously to the proof of energy conservation we easily observe that
\begin{align*}
\int_{\mathcal{D}} |\nabla \partial_{tt} u(t)|^2 + \int_{\mathcal{D}} V |\partial_{tt} u(t)|^2 
= \int_{\mathcal{D}} |\nabla \partial_{tt} u(0)|^2 + \int_{\mathcal{D}} V |\partial_{tt} u(0)|^2. 
\end{align*}
With \eqref{equality-dtt-u-0} we conclude that the right-hand side is well-defined and bounded for rough potentials $V_d$. 

To summarize our findings, we could verify 
\begin{align*}
\partial_{t} u \in L^{\infty}(H^2) \qquad \mbox{and} \qquad \partial_{tt} u \in L^{\infty}(H^1),
\end{align*}
however, any regularity assumptions of higher order seem to contradict the setting of a discontinuous potential.

$\\$
In the nonlinear case, similar arguments can be used. However, the calculations become significantly more technical since we typically do no longer have the conservation properties such as $\| \partial_t^{(k)} u \|_{L^\infty(0,T;L^2(\mathcal{D}))} = \| \partial_t^{(k)} u(0) \|_{L^2(\mathcal{D})}$ for $k\ge1$. Still for small times it is possible to show an inequality which plays a similar role. For instance, in the case of the cubic nonlinearity $\gamma(|u|^2)u=\beta |u|^2u$ (where $\beta>0$ is a parameter that characterizes the type and the number of particles) it is possible to show that there exists a minimum time $T>0$ and constant $c_T >0$ (independent of the regularity of the potential), such that
$$
\| \partial_t^{(k)} u \|_{L^\infty(0,T;L^2(\mathcal{D}))} \le c_T \| \partial_t^{(k)} u(0) \|_{L^2(\mathcal{D})}
$$
for $k \in \mathbb{N}$ and provided that $u$ is sufficiently smooth. With this it is possible to proceed in a similar way as in the linear case and one can draw the same conclusions.

\end{appendix}

\end{document}